\documentclass[onefignum, onetabnum, dvipsnames]{siamart220329}

\usepackage{lipsum}
\usepackage{graphicx}
\usepackage{epstopdf}
\usepackage{algpseudocode}
\usepackage{tensor}
\usepackage{xr-hyper}
\usepackage{hyperref}
\usepackage{cleveref}
\usepackage{framed}
\usepackage{amsmath}
\usepackage{amssymb}
\usepackage{amsfonts}
\usepackage{accents}
\usepackage{float} 
\usepackage{mathrsfs}
\usepackage{array}
\usepackage{mathtools} 
\usepackage{verbatim} 
\usepackage{caption}
\usepackage{booktabs}
\usepackage{bbm} 
\usepackage{tikz}
\usepackage{tikz-cd}
\usepackage{pgfplots}
\usetikzlibrary{shapes, arrows}
\usetikzlibrary{arrows.meta}
\usetikzlibrary{decorations}
\usetikzlibrary{decorations.text}
\usetikzlibrary{backgrounds}
\usetikzlibrary{calc}
\usetikzlibrary{animations}
\usetikzlibrary{calc, shapes, angles, quotes, patterns, decorations.pathreplacing}
\usepackage{subcaption}
\usepackage{esint}
\usepackage{dsfont}

\ifpdf
    \DeclareGraphicsExtensions{.eps, .pdf, .png, .jpg}
\else
    \DeclareGraphicsExtensions{.eps}
\fi

\newsiamremark{remark}{Remark}
\newsiamremark{hypothesis}{Hypothesis}
\crefname{hypothesis}{Hypothesis}{Hypotheses}
\newsiamthm{claim}{Claim}

\headers{\lowercase{$hp$}-stable mixed elasticity}{F.~R.~A.~Aznaran, K.~Hu, and C.~Parker}

\title{uniformly \lowercase{$hp$}-stable elements for the elasticity complex\thanks{Submitted to the editors DATE.
\funding{This work was supported by the UK Engineering and Physical Sciences Research Council through research grant EP/W522582/1, the National Science Foundation under Award No.~DMS-2201487, and a Royal Society University Research Fellowship (URF$\backslash${\rm R}1$\backslash$221398).}}}

\author{Francis R.~A.~Aznaran\thanks{Department of Applied and Computational Mathematics and Statistics, University of Notre Dame, Notre Dame, IN 46556 (\email{faznaran@nd.edu}).}
\and Kaibo Hu\thanks{School of Mathematics, University of Edinburgh, James Clerk Maxwell Building, Peter Guthrie Tait Road, Edinburgh, EH9 3FD, UK (\email{kaibo.hu@ed.ac.uk}).}
\and Charles Parker\thanks{Mathematical Institute, University of Oxford, Oxford, OX2 6GG, UK (\email{parker@maths.ox.ac.uk}).}}

\usepackage{enumitem}
\usepackage{bm}

\usepackage{amsopn}

\DeclareMathOperator{\grad}{\mathbf{grad}}
\DeclareMathOperator{\dive}{div}
\DeclareMathOperator{\spann}{span}

\DeclareMathOperator{\trace}{tr}

\DeclareMathOperator{\Div}{div}
\DeclareMathOperator{\curl}{curl}

\DeclareMathOperator{\airy}{airy}

\DeclareMathOperator{\sym}{sym}

\newcommand{\dx}{\mathrm{d}x}

\newcommand{\dy}{\mathrm{d}y}
\newcommand{\dt}{\mathrm{d}t}
\newcommand{\ds}{\mathrm{d}s}

\newcommand{\D}{\mathrm{d}}

\newcommand{\be}{\mathbf{e}}

\newcommand{\bn}{\mathbf{n}}

\newcommand{\bt}{\mathbf{t}}

\newcommand{\bx}{\mathbf{x}}

\usepackage{xfrac}
\newcommand{\vertiii}[1]{{\left\vert\kern-0.25ex\left\vert\kern-0.25ex\left\vert #1\right\vert\kern-0.25ex\right\vert\kern-0.25ex\right\vert}}

\usepackage{nicematrix}
\makeatletter
\@ifpackagelater{nicematrix}{2022/01/23}{
    \newcommand{\nicemat}{\begin{bNiceArray}{ccc|cc}[margin]}
    \newcommand{\nicediscretemat}{\begin{bNiceArray}{ccc|cc}[left-margin = .6em, right-margin = 1.15em]}
    \newcommand{\nicevec}{\begin{bNiceArray}{c}[left-margin = .4em, right-margin = .4em]}
}{
    \newcommand{\nicemat}{\begin{bNiceArray}{CCC|CC}[margin]}
    \newcommand{\nicediscretemat}{\begin{bNiceArray}{CCC|CC}[left-margin = .6em, right-margin = 1.15em]}
    \newcommand{\nicevec}{\begin{bNiceArray}{C}[left-margin = .6em, right-margin = .6em]}
}

\def\XXint#1#2#3{{\setbox0 = \hbox{$#1{#2#3}{\int}$ }
\vcenter{\hbox{$#2#3$ }}\kern-.6\wd0}}

\DeclareFontFamily{U}{matha}{\hyphenchar\font45}
\DeclareFontShape{U}{matha}{m}{n}{
      <5> <6> <7> <8> <9> <10> gen * matha
      <10.95> matha10 <12> <14.4> <17.28> <20.74> <24.88> matha12
      }{}
\DeclareSymbolFont{matha}{U}{matha}{m}{n}
\DeclareFontFamily{U}{mathx}{\hyphenchar\font45}
\DeclareFontShape{U}{mathx}{m}{n}{
      <5> <6> <7> <8> <9> <10>
      <10.95> <12> <14.4> <17.28> <20.74> <24.88>
      mathx10
      }{}
\DeclareSymbolFont{mathx}{U}{mathx}{m}{n}
\DeclareMathSymbol{\obot}{2}{matha}{"6B}
\DeclareMathSymbol{\bigobot}{1}{mathx}{"CB}

\usepackage{todonotes}

\definecolor{oxfordblue}{RGB}{4, 30, 66}

\newlength{\leftstackrelawd}
\newlength{\leftstackrelbwd}
\def\leftstackrel#1#2{\settowidth{\leftstackrelawd}
    {${{}^{#1}}$}\settowidth{\leftstackrelbwd}{$#2$}
    \addtolength{\leftstackrelawd}{-\leftstackrelbwd}
    \leavevmode\ifthenelse{\lengthtest{\leftstackrelawd>0pt}}
    {\kern-.5\leftstackrelawd}{}\mathrel{\mathop{#2}\limits^{#1}}}

\newcommand{\hdiv}[2]{H(\dive, #1; #2)}
\newcommand{\hdivz}[2]{H_{0}(\dive, #1; #2)}
\newcommand{\hdivgamma}[2]{H_{\Gamma}(\dive, #1; #2)}
\newcommand{\ltwomodrm}{\check{L}^2}
\newcommand{\rigidbodies}{\mathrm{RM}}

\newcommand{\htwospace}[0]{Q}
\newcommand{\hdivspace}[0]{\Sigma}
\newcommand{\ltwospace}[0]{V}

\renewcommand\ker{\mathcal{N}}
\newcommand\sskw{\operatorname{sskw}}
\newcommand\mskw{\operatorname{mskw}}
\renewcommand\div{\operatorname{div}}
\newcommand\ran{\mathcal{R}}

\usepackage{mathrsfs}
\usepackage{tikz-cd}

\ifpdf
\hypersetup{
  pdftitle = {A uniformly $hp$-stable element for the stress complex},
  pdfauthor = {F.~R.~A.~Aznaran, K.~Hu, and C.~Parker}
}
\fi

\begin{document}

\maketitle

\begin{abstract}
For the discretization of symmetric, divergence-conforming stress tensors in continuum mechanics, 
we prove inf-sup stability bounds which are uniform in polynomial degree and mesh size for the Hu--Zhang finite element in two dimensions.
This is achieved via an explicit construction of a bounded right inverse of the divergence operator, with the crucial component being the construction of bounded Poincar\'e operators for the
stress 
elasticity complex which are polynomial-preserving, in the Bernstein--Gelfand--Gelfand framework of the finite element exterior calculus.
We also construct $hp$-bounded projection operators satisfying a commuting diagram property and $hp$-stable Hodge decompositions. Numerical examples are provided.
\end{abstract}


\section{Introduction}

In many instances of the numerical discretization of continuum mechanics problems, it is useful to directly approximate not only the primary variable
describing deformation,
such as the displacement or velocity vector field, but also 
some form of stress tensor,
a (typically symmetric) matrix-valued dual variable 
whose constitutive relation to the strain (or strain rate) tensor forms a principal governing equation;
this is often known as the Hellinger--Reissner variational principle.
The
prototypical example we consider will be
the deformation of a linear elastic isotropic 
body $\Omega \subset \mathbb{R}^d$,
subject to a body force $f$,
with boundary $\Gamma$ partitioned into disjoint subsets $\Gamma_N$ and $\Gamma_D$,
which
is described by the Hellinger--Reissner principle in terms of the displacement $u$ and Cauchy stress tensor $\sigma$:
\begin{equation}\label{eq:H-R_strongform}
    \begin{aligned}
        \mathcal{A}\sigma &= \varepsilon(u) \quad &&\text{ in }\Omega, 
        \\
        \Div\sigma &= f \quad &&\text{ in }\Omega, 
        \\
        u &= u_0 \quad &&\text{ on }\Gamma_N, 
        \\
        \sigma \bn &= g \quad &&\text{ on }\Gamma_D, 
    \end{aligned}
\end{equation}
where 
the divergence of a tensor field is defined row-wise,
$\mathcal{A}$ denotes the compliance tensor
(in terms of Lam\'e parameters $\mu, \lambda > 0$)
\begin{align}\label{eq:compliance}
    \mathcal{A}_{ijkl} \sigma_{kl} := \frac{1}{2\mu} \left(\sigma_{ij} - \frac{\lambda}{2\lambda + 2\mu} (\trace \sigma) \delta_{ij} \right),
\end{align}
$\varepsilon = \sym\nabla$ the symmetric gradient (linearized strain) of a vector field, 
$u_0, g$ are displacement and traction data respectively,
and $\bn$ 
is 
the outward-pointing unit normal to $\Gamma$.
Let now $\mathbb{V} := \mathbb{R}^d$ and $\mathbb{S} := \mathbb{R}^{d\times d}_{\rm sym}$.
For illustration, consider the simple case in which
$u_0 = 0$, $g = 0$,
with pure displacement boundary conditions ($|\Gamma_D| = 0$).
Then, the weak form of~\cref{eq:H-R_strongform} 
is a saddle point problem which
seeks $(\sigma, u) \in \hdiv{\Omega}{\mathbb{S}} \times L^2(\Omega; \mathbb{V})$
such that
\begin{equation}
    \begin{aligned}\label{eq:hellinger-reissner-continuous}
        (\mathcal{A}\sigma, \tau) + (\dive \tau, u) &= 0 \qquad & &\forall~\tau \in \hdiv{\Omega}{\mathbb{S}}, \\
        (\dive \sigma, v) &= (f, v) \qquad & &\forall~v \in L^2(\Omega; \mathbb{V}),
    \end{aligned}
\end{equation}
where 
$\hdiv{\Omega}{\mathbb{S}}$ is the space of square-integrable symmetric tensor fields whose divergence is also square-integrable,
with graph norm $\|\cdot\|^2_{\Div, \Omega} := \|\cdot\|_{0, \Omega}^2 + \|\Div(\cdot)\|_{0, \Omega}^2$.
We employ standard notation for the Sobolev space $H^k(\Omega; \mathbb{X})$ (or $L^2(\Omega; \mathbb{X})$ when $k = 0$) with
codomain $\mathbb{X}$ (which is omitted when $\mathbb{X} = \mathbb{R}$), and associated norm $\|\cdot\|_{k, \Omega}$ and seminorm $|\cdot|_{k, \Omega}$ (or simply $\|\cdot\|_k, |\cdot|_k$).

The mixed dual formulation~\cref{eq:hellinger-reissner-continuous}
constitutes an alternative to the more common primal formulation of linear elasticity,
in which the stress is eliminated 
and $u\in H^1(\Omega; \mathbb{V})$ is solved for as the only unknown;
this mixed formulation is often motivated by, for example, avoiding the well-known 
phenomenon 
of volumetric locking in the incompressible regime $\lambda \gg \mu$.
Moreover,
at least formally, 
the compliance tensor~\cref{eq:compliance}
converges
in the incompressible limit $\lambda\to\infty$
to a scalar multiple of the \textit{deviatoric} operator which sends a matrix field to its traceless component. 
Thus, the 
model problem~\cref{eq:hellinger-reissner-continuous}
also 
serves as a starting point for 
implicitly- or nonlinearly-constituted solids.
In practical applications, one is often interested in computing the stress with at least as much accuracy as the displacement,
and hence in formulations in which the stress is computed directly rather than, for example, via numerical differentiation after the fact;
its direct approximation combined with $h$- and/or $p$-refinement
also helps
to resolve stress singularities~\cite{Carstensen2016a}.
Scalar quantities derived from a stress tensor, such as the von Mises and Tresca stresses, 
can be of great practical interest as part of yield criteria in plasticity theory~\cite{VonMises1913}\cite[Ch.~8.4]{Howell2009}.
More significantly, 
symmetric stress tensors 
are
often 
crucial coupling variables between different subproblems in multiphysics, such as
the Kirchhoff stress for electrical propagation in 
hyperelastic models of 
biomechanics~\cite{Cherubini2020},
the Cauchy stress for thermomechanical coupling in viscoelasticity~\cite{Malek2018},
or the viscous stress in multicomponent convection-diffusion~\cite{Aznaran2024}.

This work considers the role of the \textit{elasticity complex} for the discretization of such symmetric stress tensors.
For concreteness, consider the following 
conforming 
discretization of~\cref{eq:hellinger-reissner-continuous}:
find $(\sigma_{p}, u_{p - 1}) \in \hdivspace^{p} \times \ltwospace^{p - 1}$ such that
\begin{equation}
    \begin{aligned}\label{eq:hellinger-reissner-fem}
        (\mathcal{A}\sigma_{p}, \tau) + (\dive \tau, u_{p - 1}) &= 0 \qquad & &\forall~\tau \in \hdivspace^{p}, \\
        (\dive \sigma_{p}, v) &= (f, v) \qquad & &\forall~v \in \ltwospace^{p - 1},
    \end{aligned}
\end{equation}
where 
$\hdivspace^{p} \subset \hdiv{\Omega}{\mathbb{S}}$ and $\ltwospace^{p-1} \subset L^2(\Omega; \mathbb{V})$
are subspaces of the piecewise polynomials of total degree $p$ and $p - 1$ respectively, on a triangulation of the domain $\Omega$ with maximal element diameter $h$,
and $(\cdot, \cdot)_\Omega$ (or just $(\cdot, \cdot)$) denotes the $L^2(\Omega; \mathbb{X})$ inner product.
The symmetry of the stress tensor 
(which is 
typically 
equivalent to the conservation of angular momentum)
poses challenges to the construction of stable finite elements for 
problems involving an $H(\Div; \mathbb{S})$ stress.
Specifically, the  symmetry, 
unisolvency, $\Div$-conformity, and inf-sup compatibility of the pair 
$\hdivspace^{p}$ and $\ltwospace^{p - 1}$ 
in~\cref{eq:hellinger-reissner-fem}
are difficult to fulfill all at the same time, particularly on arbitrary meshes. 

Complexes 
have proved to be 
a popular tool for designing stable finite element methods. 
In the case of 
(for example)
pure displacement boundary conditions ($|\Gamma_D| = 0$),
the stress elasticity complex 
in two dimensions
reads 
\begin{equation}\label{eq:exact-sequence-continuous-disp}
    \begin{tikzcd}
        0 \arrow{r} & \mathcal{P}_1(\Omega) \arrow{r}{\iota} & H^2(\Omega) \arrow{r}{\airy} & \hdiv{\Omega}{\mathbb{S}} \arrow{r}{\dive} & L^2(\Omega; \mathbb{V}) \arrow{r} & 0,
    \end{tikzcd}
\end{equation}
where 
$\mathcal{P}_k(\Omega; \mathbb{X})$ denotes the space of $\mathbb{X}$-valued polynomials of total degree at most $k\geq 0$ on $\Omega$,
$\iota$ denotes the inclusion map, 
and the $\airy$ operator assigns to a scalar potential $q$ its matrix-valued Airy stress function, the cofactor of its Hessian,
\begin{equation*}
    \airy q := \begin{pmatrix} \partial^2_{yy} q & -\partial_{xy}^2 q \\ -\partial_{xy}^2 q & \partial_{xx}^2 q \end{pmatrix},
\end{equation*}
which is identically symmetric and divergence-free 
(and sometimes denoted $\curl\curl$).
Observe for
the sequence~\cref{eq:exact-sequence-continuous-disp} 
that each operator vanishes on the image of the previous, 
but on simply connected domains, the sequence
is exact~\cite{Arnold2002}, which means that the image of each operator is precisely the kernel of the next. For example, if $\tau \in \hdiv{\Omega}{\mathbb{S}}$ is divergence-free, then in fact $\tau = \airy q$ for some $q \in H^2(\Omega)$.

To design stable finite elements, 
one may seek 
a subcomplex  of~\cref{eq:exact-sequence-continuous-disp},
\begin{equation}\label{eq:generic-discrete-sequence}
    \begin{tikzcd}
        0 \arrow{r} & \mathcal{P}_1(\Omega) \arrow{r}{\iota} & \htwospace^{p + 2} \arrow{r}{\airy} & \hdivspace^p \arrow{r}{\dive} & \ltwospace^{p - 1} \arrow{r} & 0,
    \end{tikzcd}
\end{equation}
where $Q^{p + 2} \subset H^2(\Omega)$, and discretize~\cref{eq:hellinger-reissner-continuous} using the pair $\hdivspace^p \times \ltwospace^{p - 1}$. Choosing spaces from a complex typically leads to inf-sup stable pairs. This construction proceeds analogously in the case of mixed boundary conditions, as 
we will discuss
in~\cref{sec:fe-spaces}.
Our main focus will be on robustness with respect to $p$, since $p$- and $hp$-version finite element methods deliver superior convergence rates for both smooth problems and problems with singularities~\cite{Babuska1994} and are better suited for modern (super)computer architectures~\cite{Kolev2021}. As such, we suppress notational dependence 
on $h$.

Indeed, while 
many works on $H(\Div; \mathbb{S})$ elements are stated for arbitrary-order elements, 
to our knowledge the stability (or otherwise) of such methods with respect to polynomial degree $p$ has not been considered.
In this work, we study the $hp$ properties of a choice of discrete exact sequence~\cref{eq:generic-discrete-sequence}
in two dimensions.
We
construct bounded right inverses to the divergence operator from the discrete displacement space to the discrete stress space, which are bounded independently of polynomial degree $p$, for a specific choice of $\Sigma^p$ and $V^{p - 1}$. 
An immediate consequence of this 
statement,
to our knowledge the first result of its kind, 
is that the finite element pair is uniformly $hp$-stable for the elasticity problem~\cref{eq:hellinger-reissner-fem},
as well as for other problems requiring the same inf-sup condition, such as the linear Reissner--Mindlin plate~\cite{Sky2023}. 

The proof relies on the construction of Poincar\'{e} operators using the Bernstein--Gelfand--Gelfand (BGG) framework of finite element exterior calculus~\cite{Arnold2021} and on the existence of polynomial extension operators~\cite{Parker2023}. Using the bounded inverse of the divergence operator, we 
also 
construct projection operators,
connecting the original complex to the discrete complex,
that satisfy a commuting diagram property, 
and show that the symmetric tensor elements admit an $hp$-stable Hodge decomposition.

The stress space coincides with the triangular Hu--Zhang element~\cite{Hu2014},
which are 
one 
in a long line of attempts -- beginning with the Arnold--Winther elements~\cite{Arnold2002} -- to exactly enforce symmetry and $\Div$-conformity of 
stress tensors in continuum mechanics.
The Hu--Zhang element 
was shown using BGG to form a discrete complex of its own in~\cite{Christiansen2018}, and 
has also been applied to discretize 
the bending moment tensor for the Reissner--Mindlin plate~\cite{Sky2023}.
Although we have emphasized the broad utility of $H(\Div; \mathbb{S})$ stress elements,
the tradeoffs between 
the challenges in constructing them
have garnered much attention in the finite element literature
even just for the canonical Hellinger--Reissner problem~\cref{eq:hellinger-reissner-fem}.
Typically, one or more of 
the desired properties of the pair $\hdivspace^{p}\times\ltwospace^{p - 1}$ 
is sacrificed or modified.
For example, 
symmetry may be enforced weakly via penalization~\cite{Cai2004}
or a Lagrange multiplier~\cite{Arnold2007},
the TDNNS approach trades regularity requirements between the stress and displacement in the variational formulation~\cite{Schoeberl2007},
and
the Gopalakrishnan--Guzm\'an element in 3D~\cite{Gopalakrishnan2012} is such that weak enforcement of symmetry implies its strong enforcement, but it
requires an Alfeld-split mesh,
as do a recent work on a finite element elasticity complex on the 3D Alfeld-split~\cite{Christiansen2024} and a recent low-order contribution in 3D which generalizes 
the Johnson--Mercier element~\cite{Johnson1978}
to arbitrary dimensions~\cite{Gopalakrishnan2024}.

This work is structured as follows. In~\cref{sec:elasticity-complex}, we clarify our notation, and certain technical properties of the continuous (i.e.~infinite-dimensional) elasticity complex~\cref{eq:exact-sequence-continuous-disp}. 
Next in~\cref{sec:fe-spaces}, we define our discrete analog of this, the Hu--Zhang complex, proving local and global unisolvency of the stress space.
We summarize our three main results in~\cref{sec:main-results}, before immediately providing numerical evidence thereof in~\cref{sec:numerics}.
Proof of the principal result, the $p$-stable inversion of the divergence, is spread over~\cref{sec:poincare-element} (for a single element) and~\cref{sec:invert-div-mesh} (over an entire mesh); 
proofs the other main results, namely the existence of $hp$-bounded commuting projections and $hp$-stable Hodge decompositions, are provided in~\cref{sec:cochain}.

\section{Problem setting and the elasticity complex}\label{sec:elasticity-complex}

We assume that the domain $\Omega$ is polygonal with a finite number of polygonal holes, i.e.~there exist open simply connected polygonal domains $\{ \Omega_i \}_{i=0}^{M} \subset \mathbb{R}^2$ satisfying the following conditions: (1) $\Omega = \Omega_0 \setminus \cup_{i=1}^{M} \overline{\Omega}_i$; (2) $\overline{\Omega}_i \subset \Omega_0$ for $1 \leq i \leq M$; and (3) $\overline{\Omega}_j \cap \overline{\Omega}_k = \emptyset$ for $1 \leq j, k \leq M$ and $j \neq k$. Let $\{ A_j \}_{j=1}^N$ and $\{ \Gamma_j \}_{j=1}^N$ denote the vertices and edges of $\Gamma$, ordered counterclockwise on each connected component of $\Gamma$ as in~\cref{fig:domain-example}. We assume that the edge labels are partitioned into sets $\mathcal{D}$ and $\mathcal{N}$ such that $\mathcal{D} \cup \mathcal{N} = \{1, \ldots, N\}$, $\Gamma_D = \cup_{i \in \mathcal{D}} \Gamma_i$ and $\Gamma_N = \cup_{i \in \mathcal{N}} \Gamma_i$.

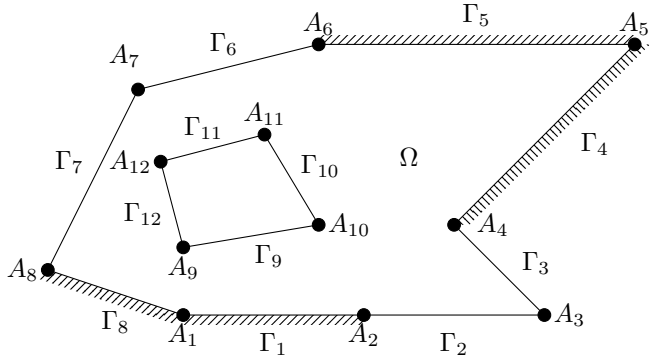
\begin{figure}[H]
    \centering
    \begin{tikzpicture}[scale=1.2]
        \coordinate (A1) at (-1, 0);
        \coordinate (A2) at (1, 0);
        \coordinate (A3) at (3, 0);
        \coordinate (A4) at (2, 1);
        \coordinate (A5) at (4, 3);
        \coordinate (A6) at (0.5, 3);
        \coordinate (A7) at (-1.5, 2.5);
        \coordinate (A8) at (-2.5, 0.5);
        \coordinate (B1) at (-1, 0.75);
        \coordinate (B2) at (0.5, 1);
        \coordinate (B3) at (-0.1, 2);
        \coordinate (B4) at (-1.25, 1.7);
        \coordinate (Olabel) at (1.5, 1.75);
        
        \tikzstyle{ground}=[fill, pattern = north east lines, draw = none, minimum width = 0.75cm, minimum height = 0.3cm, inner sep = 0pt, outer sep = 0pt]
        
        
        \draw[ground] (A1) rectangle ($(A2) + (0, -0.1)$);
        
        
        \draw[ground] (A8) -- (A1) -- ($(A1) + (-0.1, -0.1)$) -- ($(A8) + (-0.1, -0.1)$) -- (A8);
        
        
        \draw[pattern=north west lines,draw=none,minimum width=0.75cm,minimum height=0.3cm,inner sep=0pt,outer sep=0pt] (A4) -- (A5) -- ($(A5) + (0.17, 0)$) -- ($(A4) + (0.17,0)$) -- (A4);
        
        \draw[pattern=north east lines,draw=none,minimum width=0.75cm,minimum height=0.3cm,inner sep=0pt,outer sep=0pt] (A5) -- (A6) -- ($(A6) + (0, 0.1)$) -- ($(A5) + (0,0.1)$) -- (A5);
        
        \draw (A8) -- (A1) -- (A2);
        \draw (A4) -- (A5) -- (A6);
        \draw (A2) -- (A3) -- (A4);
        \draw (A6) -- (A7) -- (A8);
        
        \draw (B1) -- (B2) -- (B3) -- (B4) -- (B1);
        
        \filldraw (A1) circle (2pt) node[align=center,below]{${A}_1$};
        \filldraw (A2) circle (2pt) node[align=center,below]{${A}_2$};    
        \filldraw  (A3) circle (2pt) node[align=center,right]{${A}_3$};
        \filldraw (A4) circle (2pt) node[align=center,right]{};
        \filldraw (A5) circle (2pt) node[align=center,above]{${A}_5$};
        \filldraw (A6) circle (2pt) node[align=center,above]{${A}_6$};
        \filldraw (A7) circle (2pt) node[align=center,above]{};
        \filldraw (A8) circle (2pt) node[align=center,left]{${A}_8$};    
        
        \draw ($(A4) + (0.15,0)$)  node[align=center,right]{${A}_4$};
        \draw ($(A7) + (-0.15,0.15)$)  node[align=center,above]{${A}_7$};
        
        \filldraw (B1) circle (2pt) node[align=center,below]{${A}_9$};
        \filldraw (B2) circle (2pt) node[align=center,right]{${A}_{10}$};    
        \filldraw  (B3) circle (2pt) node[align=center,above]{${A}_{11}$};
        \filldraw (B4) circle (2pt) node[align=center,left]{${A}_{12}$};
        
        \draw ($($(A1)!0.5!(A2)$) + (0, -0.1) $) node[align=center,below]{$\Gamma_1$};
        
        \draw ($($(A2)!0.5!(A3)$) + (0, -0.1) $) node[align=center,below]{$\Gamma_2$};
        
        \draw ($($(A3)!0.5!(A4)$) + (0.15, 0.08) $) 
        node[align=center,right]{$\Gamma_3$};
        
        \draw ($($(A4)!0.5!(A5)$) + (0.3, -0.1) $) node[align=center,right]{$\Gamma_4$};
        
        \draw ($($(A5)!0.5!(A6)$) + (0.0, 0.1) $) node[align=center,above]{$\Gamma_5$};
        
        \draw ($($(A6)!0.5!(A7)$) + (-0.05, 0.25) $) node[align=center]{$\Gamma_6$};
        
        \draw ($($(A7)!0.5!(A8)$) + (0., 0.2) $) node[align=center,left]{$\Gamma_7$};
        
        \draw ($($(A8)!0.5!(A1)$) + (0., -0.1) $) node[align=center,below]{$\Gamma_8$};
        
        \draw ($($(B1)!0.5!(B2)$) + (0.2, 0.) $) node[align=center,below]{$\Gamma_9$};
        
        \draw ($($(B2)!0.5!(B3)$) + (0., 0.15) $) node[align=center,right]{$\Gamma_{10}$};
        
        \draw ($($(B3)!0.5!(B4)$) + (-0.1, 0.) $) node[align=center,above]{$\Gamma_{11}$};
        
        \draw ($($(B4)!0.5!(B1)$) + (0., -0.1) $) node[align=center,left]{$\Gamma_{12}$};
        
        \draw (Olabel) node[align=center]{$\Omega$};
    \end{tikzpicture}
    \caption{Example domain $\Omega$ and boundary partition $\mathcal{D} = \{2, 3, 6, 7, 9, 10, 11, 12\}, \mathcal{N} = \{1, 4, 5, 8\}$.}\label{fig:domain-example}
\end{figure}

\subsection{Defining the complex with boundary conditions}

Boundary conditions for a symmetric stress tensor are typically imposed on its normal component.
Consider 
therefore
the stress space $\hdivgamma{\Omega}{\mathbb{S}}$ consisting of the subspace of 
$\hdiv{\Omega}{\mathbb{S}}$
whose normal component vanishes on $\Gamma_D$ in the sense of distributions:
\begin{align}\label{eq:hdivgamma-def}
    \hdivgamma{\Omega}{\mathbb{S}} := \{ \tau \in \hdiv{\Omega}{\mathbb{S}}:\langle\tau \bn, v\rangle_{\Gamma} = 0 \ \forall~v \in H^1(\Omega; \mathbb{V}) \text{ with } v|_{\Gamma_N} = 0 \},
\end{align}
where $\langle\cdot, \cdot\rangle_{\Gamma}$ denotes the $(H^{-1/2}\times H^{1/2})(\Gamma)\text{ or }\mathbb{R}^2)$ dual pairing.
We seek to define appropriate spaces $\mathcal{P}_{1, \Gamma}(\Omega)$, $H^2_{\Gamma}(\Omega)$, and $L^2_{\Gamma}(\Omega; \mathbb{V})$ so that for $\hdivgamma{\Omega}{\mathbb{S}}$ defined as in~\cref{eq:hdivgamma-def}, the sequence
\begin{equation}\label{eq:exact-sequence-continuous-gen}
    \begin{tikzcd}
        0 \arrow{r} & \mathcal{P}_{1, \Gamma}(\Omega) \arrow{r}{\iota} & H^2_{\Gamma}(\Omega) \arrow{r}{\airy} & \hdivgamma{\Omega}{\mathbb{S}} \arrow{r}{\dive} & L^2_{\Gamma}(\Omega; \mathbb{V}) \arrow{r} & 0
    \end{tikzcd}
\end{equation}
is exact on simply connected domains. 

We first examine the image of the divergence operator acting on $\hdivgamma{\Omega}{\mathbb{S}}$. Consider the case in which Dirichlet conditions are imposed on the entire boundary ($|\Gamma| = |\Gamma_D|$),
and define the rigid body motions $\rigidbodies := \mathbb{R}^2 \oplus \spann \{ \bx^{\perp} := (-x^2, x^1)^{\top} \}$.
It is well known, for $\Omega$ simply connected, that $\rigidbodies$ is precisely the kernel of the symmetric gradient 
$\varepsilon$; 
since it is the kernel of the formal adjoint of the divergence, it is orthogonal to the image of the divergence acting on the trace-free subspace of $\hdiv{\Omega}{\mathbb{S}}$, i.e.~for any $\sigma \in \hdivgamma{\Omega}{\mathbb{S}}$ and $r\in\rigidbodies$, we have
\begin{equation}\label{eq:image-of-div-orthog}
    \int_{\Omega} r \cdot (\dive \sigma)~\dx = \int_{\partial\Omega}r\cdot\sigma\bn~\ds - \int_{\Omega}\underbrace{\nabla r:\sigma}_{\let\scriptstyle\displaystyle\substack{= \varepsilon(r):\sigma}}~\dx = 0.
\end{equation}
Moreover by choosing smooth and compactly supported $\sigma\in C^\infty_0(\Omega; \mathbb{S})\subset\hdivgamma{\Omega}{\mathbb{S}}$ and smooth $\tilde{r}\in C^\infty(\Omega; \mathbb{V})$, it is clear from~\cref{eq:image-of-div-orthog} that $\tilde{r}\in(\Div\hdivgamma{\Omega}{\mathbb{S}})^\perp$ if and only if $\varepsilon(\tilde{r}) = 0$ (where $\perp$ denotes the orthogonal complement),
and hence $\Div\hdivgamma{\Omega}{\mathbb{S}}$ is dense in the orthogonal complement of $\rigidbodies$ within $L^2(\Omega; \mathbb{V})$.

If however $|\Gamma_N| > 0$, then it is not clear that there are any further conditions on $\dive \sigma$. Consequently, we define $L^2_{\Gamma}(\Omega)$ as follows: 
\begin{align}\label{eq:l2gamma-def}
    L_{\Gamma}^2(\Omega; \mathbb{V}) &:= \begin{cases}
        \{ v \in L^2(\Omega; \mathbb{V}):(v, r) = 0~\forall~r \in \rigidbodies \} \simeq L^2(\Omega;\mathbb{V})/\rigidbodies & \text{if } |\Gamma| = |\Gamma_D|, \\ 
        L^2(\Omega; \mathbb{V}) & \text{otherwise}.
    \end{cases}
\end{align}
Later in~\cref{thm:h1-inversion-div-bc} we will show that in fact $\Div:\hdivgamma{\Omega}{\mathbb{S}}\to L^2_{\Gamma}(\Omega; \mathbb{V})$ is surjective in both of these cases.

We now want to define $H^2_{\Gamma}(\Omega) \subset H^2(\Omega)$ so that $\airy H^2_{\Gamma}(\Omega) \subset \hdivgamma{\Omega}{\mathbb{S}}$. If $r \in H^2(\Omega)$ and $(\airy r)\bn|_{\Gamma_D} = 0$, then
\begin{equation*}
    \partial_{\bt} (\grad r)^{\perp}|_{\Gamma_D} =  \partial_{\bt} ( \partial_{\bn} r \bt - \partial_{\bt} r \bn )|_{\Gamma_D} = (\airy r) \bn|_{\Gamma_D} = 0,
\end{equation*}
and so $\grad r$ is constant on $\Gamma_D$. If we denote by $\{ \Gamma_{D, j} : 0 \leq j \leq J_D \}$ the connected components of $\Gamma_D$, then
\begin{align}\label{eq:h2gamma-bcs-alt}
    r|_{\Gamma_{D, j}} \in \mathcal{P}_1(\Gamma_{D, j}), \quad \text{and} \quad \partial_{\bn} r|_{\Gamma_{D, j}} \in \mathcal{P}_0(\Gamma_{D, j}), \qquad 0 \leq j \leq J_D. 
\end{align}
Moreover, there exists an affine function $\ell$ such that $q = r - \ell$ satisfies $\airy q = \airy r \in \hdivgamma{\Omega}{\mathbb{S}}$ and the following boundary conditions:
\begin{align}\label{eq:h2gamma-bcs}
    q|_{\Gamma_{D, 0}} = \partial_{\bn} q|_{\Gamma_{D, 0}} = 0, \quad 
    q|_{\Gamma_{D, j}} \in \mathcal{P}_1(\Gamma_{D, j}) \quad \text{and} \quad \partial_{\bn} q|_{\Gamma_{D, j}} \in \mathcal{P}_0(\Gamma_{D, j}), \ \ 1 \leq j \leq J_D. 
\end{align}
In particular, if $\Omega$ is simply connected and $\Gamma = \Gamma_D$, then $q \in H^2_0(\Omega)$. Consequently, we define $H^2_{\Gamma}(\Omega)$ as follows:
\begin{equation}\label{eq:h2gamma-def}
    H_{\Gamma}^2(\Omega) := \{ q \in H^2(\Omega): \text{$q$ satisfies~\cref{eq:h2gamma-bcs}} \}.
\end{equation}
Note that the choice of $\Gamma_{D, 0}$ from the $J_D + 1$ connected components is arbitrary.
Note also that if $\Gamma = \Gamma_D$ and $\Omega$ is not simply connected (so that $\Gamma$ is multiply connected), then~\cref{eq:h2gamma-def}
does not coincide with the na\"ive definition of $H^2_{\Gamma}(\Omega)$, i.e.~$H^2_0(\Omega) := \{ q \in H^2(\Omega) : q|_{\Gamma} = \partial_{\bn} q|_{\Gamma} = 0 \} \subsetneq H^2_{\Gamma}(\Omega)$.
One encounters a similar situation with the stream functions of divergence-free $H^1_0(\Omega; \mathbb{V})$ vector fields: If $v \in H^1_0(\Omega; \mathbb{V})$ satisfies $\dive v \equiv 0$, then $v = \curl \psi$ for some $\psi \in H^2(\Omega)$, where the trace of $\psi$ is constant on the connected components of $\Gamma$~\cite[p. 42 eq. (3.11)]{Girault1986}.

We finally turn to $\mathcal{P}_{1, \Gamma}(\Omega)$. If $q \in H^2_{\Gamma}(\Omega)$ satisfies $\airy q \equiv 0$, then $q \in \mathcal{P}_1(\Omega)$ and $q|_{\Gamma_{D, 0}} = \partial_{\bn} q|_{\Gamma_{D, 0}} = 0$. Thus, $q \equiv 0$ provided that $|\Gamma_D| > 0$, which leads to the following choice of $\mathcal{P}_{1, \Gamma}(\Omega)$:
\begin{align}\label{eq:left-exact-seq-space-bc}
    \mathcal{P}_{1,\Gamma}(\Omega) &:= \begin{cases}
        \mathcal{P}_{1}(\Omega) & \text{if } |\Gamma_N| = |\Gamma|, \\
        \{0\} & \text{otherwise}.
    \end{cases}
\end{align}
We will see in~\cref{lem:exact-sequence-continuous-disp} that the sequence~\cref{eq:exact-sequence-continuous-gen} is exact on simply connected domains. 

\begin{remark}
    One may alternatively take $H^2_{\Gamma}(\Omega)$ to be the space $\{ r \in H^2(\Omega) : r \text{ satisfies~\cref{eq:h2gamma-bcs-alt}} \}$ and $\mathcal{P}_{1, \Gamma}(\Omega)$ to simply be $\mathcal{P}_1(\Omega)$. Here, we opt for the choices~\cref{eq:h2gamma-def,eq:left-exact-seq-space-bc} so that $H^2_{\Gamma}(\Omega) = H^2_0(\Omega)$ when $\Omega$ is simply connected and $|\Gamma| = |\Gamma_D|$.
\end{remark}
With all of the spaces in the complex~\cref{eq:exact-sequence-continuous-gen} in hand, we now turn to properties of the complex.

\subsection{Properties of the continuous elasticity complex}

The first key property is that the divergence operator from $\hdivgamma{\Omega}{\mathbb{S}}$~\cref{eq:hdivgamma-def} to $L^2_{\Gamma}(\Omega)$~\cref{eq:l2gamma-def} is surjective and admits a bounded right inverse from $L^2_{\Gamma}(\Omega)$ to $H^1(\Omega; \mathbb{S}) \cap \hdivgamma{\Omega}{\mathbb{S}}$.
\begin{theorem}\label{thm:h1-inversion-div-bc}
    \emph{($H^1(\Omega; \mathbb{S})$ inversion of the divergence.)}
    For every $u \in L^2_{\Gamma}(\Omega; \mathbb{V})$, there exists $\sigma \in H^1(\Omega; \mathbb{S}) \cap \hdivgamma{\Omega}{\mathbb{S}}$ such that
    \begin{equation}\label{eq:h1-inversion-div-bc}
        \dive \sigma = u \quad \text{and} \quad \|\sigma\|_{1, \Omega} \leq C \|u\|_{0, \Omega},
    \end{equation}
    where $C$ is independent of $u$.
\end{theorem}
The proof of~\cref{thm:h1-inversion-div-bc} appears in~\cref{sec:proof-of-h1-inversion-div-bc}. The extra regularity on the tensor $\sigma$ in~\cref{thm:h1-inversion-div-bc} plays a crucial role in the ensuing analysis.

We now turn to the dimension of the space of \textit{harmonic forms}
\begin{equation}\label{eq:harmonic-forms-def}
    \mathfrak{H}_{\Gamma}(\Omega) := \{ \phi \in \hdivgamma{\Omega}{\mathbb{S}} : \dive \phi \equiv 0 \text{ and } (\phi, \airy q)_{\dive, \Omega} = 0 \ \forall~q \in H^2_{\Gamma}(\Omega) \}, 
\end{equation}
which consist of divergence-free tensors which cannot be expressed as the airy of some $H^2_{\Gamma}(\Omega)$ function. To this end, we define the index set $\mathfrak{I}$ consisting of connected components of the boundary $\partial \Omega$ which intersect $\Gamma_N$, and $\mathfrak{I}^*$ with all but one element in $\mathfrak{I}$:
\begin{equation}\label{eq:index-set-def}
    \mathfrak{I} := \{ 0 \leq m \leq M : |\partial \Omega_m \cap \Gamma_N| > 0 \} \quad \text{and} \quad \mathfrak{I}^* := \begin{cases}
        \emptyset & \text{if } \mathfrak{I} = \emptyset, \\
        \mathfrak{I} \setminus \{ \min \mathfrak{I} \} & \text{otherwise}.
    \end{cases}.
\end{equation}
Then the following result shows that $\dim \mathfrak{H}_{\Gamma}(\Omega) = 3|\mathfrak{I}^*|$:
\begin{lemma}\label{lem:exact-sequence-continuous-disp}
    There holds $\dim \mathfrak{H}_{\Gamma}(\Omega) = 3 |\mathfrak{I}^*|$ and the kernel of $\airy : H^2_{\Gamma}(\Omega) \to \hdivgamma{\Omega}{\mathbb{S}}$ is $\mathcal{P}_{1, \Gamma}(\Omega)$. Consequently, if $|\mathfrak{I}^*| = 0$, then the sequence~\cref{eq:exact-sequence-continuous-gen} is exact.
\end{lemma}
One immediate consequence of~\cref{lem:exact-sequence-continuous-disp} is that the sequence~\cref{eq:exact-sequence-continuous-gen} is exact on simply connected domains. The dimension of $\mathfrak{H}_{\Gamma}(\Omega)$ appears to be a novel result for mixed boundary conditions ($|\Gamma_D| > 0$) and for $\Omega$ not simply connected. The proof of~\cref{lem:exact-sequence-continuous-disp} appears in~\cref{sec:proof-of-exact-sequence-continuous-disp}.

\section{Finite element spaces}\label{sec:fe-spaces}

Let $\mathcal{T}$ be a shape-regular partitioning of the domain $\Omega$ into triangles such that the nonempty intersection of any two distinct elements from $\mathcal{T}$ is either a single common vertex or a single common edge. The parameter $h > 0$ will denote the diameter of the largest element: $h := \max_{K \in \mathcal{T}} h_K$ and $h_K := \mathrm{diam}(K)$.  We further assume that mesh vertices are placed on the intersection of $\overline{\Gamma_D}$ and $\overline{\Gamma_N}$ (see e.g.~$A_2$ in~\cref{fig:domain-example}).

We use the following spaces for $p \geq 2$:
\begin{align*}
    \htwospace^{p} &:= \{ q \in H^2(\Omega):q|_{K} \in \mathcal{P}_{p}(K) \ \forall~K \in \mathcal{T}, \text{ $q$ is $C^2$ at element vertices} \}, \\
    \hdivspace^{p} &:= \{ \sigma \in \hdiv{\Omega}{\mathbb{S}}:\sigma|_{K} \in \mathcal{P}_p(K; \mathbb{S}) \ \forall~K \in \mathcal{T}, \text{ $\sigma$ is $C^0$ at element vertices} \}, \\
    \ltwospace^{p} &:= \{ v \in L^2(\Omega; \mathbb{V}):v|_{K} \in \mathcal{P}_{p}(K; \mathbb{V}) \ \forall~K \in \mathcal{T} \}.
\end{align*}
The spaces $\htwospace^{p}$ and $\ltwospace^{p}$ are well-known: $\htwospace^{p}$, $p \geq 5$, is the  Argyris/TUBA finite element~\cite{Argyris1968} of arbitrary order and $\ltwospace^{p}$ is the vector discontinuous Lagrange element of degree $p$. The space $\hdivspace^{p}$, $p \geq 3$, which is a superset of the conforming Arnold--Winther stress element~\cite{Arnold2002} (constrained to have divergence of degree $p - 2$), is an alternative characterization of the triangular Hu--Zhang element~\cite{Hu2014}, as we show in~\cref{sec:hu-zhang}. 

The corresponding finite element spaces with boundary conditions are simply
\begin{align}\label{eq:fem-spaces-bcs}
    \htwospace_{\Gamma}^{p} := \htwospace^{p} \cap H_{\Gamma}^2(\Omega), \quad \hdivspace_{\Gamma}^{p} := \hdivspace^{p} \cap \hdivgamma{\Omega}{\mathbb{S}}, \quad \text{and} \quad \ltwospace_{\Gamma}^{p} := \ltwospace^{p} \cap L^2_{\Gamma}(\Omega; \mathbb{V}),
\end{align}
where $H^2_{\Gamma}(\Omega)$, $\hdivgamma{\Omega}{\mathbb{S}}$, and $L^2_{\Gamma}(\Omega)$ are defined in~\cref{eq:h2gamma-def},~\cref{eq:hdivgamma-def}, and~\cref{eq:l2gamma-def}, respectively. For $p \geq 3$, the spaces form the following conforming subcomplex of~\cref{eq:exact-sequence-continuous-gen}:
\begin{equation}\label{eq:exact-sequence-fem}
    \begin{tikzcd}
        0 \arrow{r} & \mathcal{P}_{1,\Gamma}(\Omega) \arrow{r}{\iota} & \htwospace_{\Gamma}^{p+2} \arrow{r}{\airy} & \hdivspace_{\Gamma}^{p} \arrow{r}{\dive}  & \ltwospace_{\Gamma}^{p-1} \arrow{r} & 0.
    \end{tikzcd} 
\end{equation}

\subsection{Local degrees of freedom for the stress space}

\begin{lemma}\label{lem:dim-interior-stress}
    \emph{(Dimension of the normally vanishing local tensor space.)}
    For $p \geq 0$ and $K \in \mathcal{T}$, let
    \begin{equation*}
        \hdivspace_{0}^p(K; \mathbb{S}) := \mathcal{P}_{p}(K; \mathbb{S}) \cap \hdivz{K}{\mathbb{S}} = \{ \tau \in \mathcal{P}_p(K; \mathbb{S}):\tau \bn|_{\partial K} = 0 \}.
	\end{equation*}
    Then $\dim \hdivspace_{0}^p(K; \mathbb{S}) = \frac{3}{2} p(p - 1)$.
\end{lemma}
\begin{proof}
    We have $\dim\mathcal{P}_p(K; \mathbb{S}) = (\dim\mathcal{P}_p(K))\dim\mathbb{S} = \frac32(p + 2)(p + 1)$.
    Let $\tau$ be a member of this superspace $\mathcal{P}_p(K; \mathbb{S})$; we count the Lagrange degrees of freedom
    (DOFs)
    which must vanish on $\tau$ if it moreover lies in $\hdivspace_0^p(K; \mathbb{S})$.
    \begin{itemize}[leftmargin = 4mm]
        \item At each vertex, the normal components of $\tau$ associated with the two adjoining edges must vanish; since these normals are linearly independent, in fact $\tau$ must vanish. This means the evaluations of each of its 3 components must vanish at each vertex (9 DOFs).
        \item On each edge $\gamma$, we have $\tau\bn\in\mathcal{P}_p(\gamma; \mathbb{V})$; if $\tau\bn = 0$ on $\gamma$, then it must vanish under the $2(p + 1)$ DOFs associated with 
    $\mathcal{P}_p(\gamma; \mathbb{V})$.
    Not counting the two vertex evaluations of each component at the ends of each $\gamma$, this gives $3\cdot 2(p + 1 - 2) = 6(p - 1)$ DOFs.
    \end{itemize}
    This shows $\dim\hdivspace_0^p(K; \mathbb{S})\leq \frac32(p + 2)(p + 1) - (9 + 6(p - 1)) = \frac{3}{2} p(p - 1)$.

    Let now $\{\bt_i\}_{i = 1}^3$ denote the tangents to the edges of $K$, 
    and let $\{\lambda_i\}_{i = 1}^3$ denote the barycentric coordinates on $K$.
    Consider the basis of $\mathbb{S}$ given by 
    $\{\bt_i\otimes\bt_i\}_{i = 1}^3$.
    It is easy to reason that $\hdivspace^p_0(K; \mathbb{S})$ contains the space
    \begin{equation}\label{eq:interior-space}
        \lambda_2\lambda_3\mathcal{P}_{p - 2}(K)(\bt_1\otimes\bt_1) \bigoplus \lambda_1\lambda_3\mathcal{P}_{p - 2}(K)(\bt_2\otimes\bt_2) \bigoplus \lambda_1\lambda_2\mathcal{P}_{p - 2}(K)(\bt_3\otimes\bt_3),
    \end{equation}
    which has dimension $3\dim\mathcal{P}_{p - 2}(K) = \frac32p(p - 1)$.
\end{proof}

The proof of~\cref{lem:dim-interior-stress} shows that in fact $\hdivspace_0^p(K; \mathbb{S})$ coincides with the space in~\cref{eq:interior-space}.
Now, let $\mathcal{V}$ and $\mathcal{E}$ denote the set of element vertices and edges. Similarly, given an element $K \in \mathcal{T}$, let $\mathcal{V}_K$ denote its vertices and $\mathcal{E}_K$ its edges. 

\begin{lemma}\label{lem:hdiv-local-dofs}
    \emph{(Unisolvency for the local tensor space.)}
    For $p \geq 1$, the space $\mathcal{P}_p(K; \mathbb{S})$ is unisolvent with respect to the following DOFs for $\sigma \in \mathcal{P}_p(K; \mathbb{S})$:
    \begin{enumerate}[label = (\roman*)]
        \item 3 values per vertex $\bx\in \mathcal{V}_K$: $\sigma(\bx)$,
        \item $2(p-1)$ values per edge $\gamma \in \mathcal{E}_K$:
        \begin{align}\label{eq:edge-dofs}
            \int_{\gamma} (\sigma \bn) \cdot r~\ds \qquad \forall~r \in \mathcal{P}_{p-2}(\gamma; \mathbb{V}),
        \end{align}
        \item $\frac{3}{2} p(p-1)$ values:
        \begin{align}\label{eq:internal-dofs}
            \int_{K} \sigma : \tau \ \dx \qquad \forall~\tau \in \hdivspace_0^{p}(K; \mathbb{S}).
        \end{align}
    \end{enumerate}
\end{lemma}
\begin{proof}
    The $p = 1$ case is clear, noting in this case that $\hdivspace_0^p(K; \mathbb{S}) = 0$ and by convention $\mathcal{P}_{-1} = \emptyset$, so now assume the DOFs vanish on some $\sigma\in\mathcal{P}_p(K; \mathbb{S})$ with $p\geq 2$.
    Since $\sigma$ vanishes at each vertex $\bx\in\mathcal{V}_K$, then on each edge $\gamma\in\mathcal{E}_K$ connecting vertices $k, j$ of $K$, we have $\sigma\in \lambda_k\lambda_j\mathcal{P}_{p - 2}(\gamma; \mathbb{S})$,
    and in particular we may write $\sigma\bn = \lambda_k\lambda_j\omega$ for some $\omega\in\mathcal{P}_{p - 2}(\gamma; \mathbb{V})$.
    Choosing $r = \omega$ in~\cref{eq:edge-dofs}, 
    we must have $\sigma\bn = 0$ on each $\gamma$ and hence on $\partial K$.
    Thus $\sigma\in\hdivspace_0^p(K; \mathbb{S})$, so by orthogonality to the same space~\cref{eq:internal-dofs} we have $\sigma = 0$.
\end{proof}

The local degrees of freedom for each of the spaces in the discrete complex  \cref{eq:exact-sequence-fem} as illustrated in~\cref{fig:hu-zhang}. Observe that
that vertex basis functions for the Hu--Zhang space (dual to their $C^0$ degrees of freedom) 
arise precisely from 
the $\airy$ of
those in the preceding Argyris space (which are dual to its $C^2$ degrees of freedom).

\begin{figure}[htb]
	\scalebox{0.9}{
		\centering
		\begin{tikzpicture}[scale = 3.0]
			\draw[line width = 0.5mm, color = OliveGreen] (0,0) -- (1, 0) -- (0.5, 0.866) -- cycle;
			\foreach \i/\j in {0/0, 1/0, 0.5/0.866}{
				\draw[color = OliveGreen, fill = OliveGreen] (\i, \j) circle (0.02);
				\draw[color = OliveGreen] (\i, \j) circle (0.05);
				\draw[color = OliveGreen] (\i, \j) circle (0.08);
			}
			\foreach \i/\j/\n/\t in {0.5/0.0/0.0/-1,
				0.75/0.433/0.866/0.5,
				0.25/0.433/-0.866/0.5}{
				\draw[line width = 0.4mm, -Latex, color = OliveGreen] (\i, \j) -- (\i + \n/4, \j + \t/4);
			}
			\path[draw, -{Latex}] (1, 0.6) -- node [above] {$\airy$} (1.8, 0.6);
			\begin{scope}[shift = {(1.8, 0)}]
				\draw[line width = 0.5mm, color = blue] (0,0) -- (1, 0) -- (0.5, 0.866) -- cycle;
				\foreach \i/\j in {0/0, 1/0, 0.5/0.866}{
					\foreach \dx/\dy in {0.02/0, -0.02/0, 0/0.03}{
						\draw[color = blue, fill = blue] (\i + \dx, \j + \dy) circle (0.02);
					}
				}
				\foreach \i/\j/\n/\t in {0.25/0.0/0.0/-1,
					0.5/0.0/0.0/-1,
					0.75/0.0/0.0/-1,
					0.875/0.2083/0.866/0.5,
					0.75/0.4167/0.866/0.5,
					0.625/0.625/0.866/0.5,
					0.125/0.2083/-0.866/0.5,
					0.25/0.4167/-0.866/0.5,
					0.375/0.625/-0.866/0.5
				}{
					\draw[line width = 0.4mm, -implies, double, double distance = 1.5mm, color = blue] (\i, \j) -- (\i + \n/4, \j + \t/4);
				}
				\foreach \dx/\dy in {0.02/0.0, -0.02/0.0, 0.0/0.04
					,0.06/0.0, -0.06/0.0, 0.02/0.08, 0.04/0.04, -0.04/0.04, -0.02/0.08}{
					\draw[color = blue, fill = white] (0.5 + \dx, 0.2887 + \dy) circle (0.02);
				}
				\path[draw, -{Latex}] (1, 0.6) -- node [above] {$\Div$} (1.8, 0.6);
		\end{scope}
		\begin{scope}[shift = {(3.4, 0)}]
				\draw[line width = 0.5mm, color = red] (0, 0) -- (1, 0) -- (0.5, 0.866) -- cycle;
				\foreach \dx in {0.02, -0.02}{
					\foreach \a/\b in {0.375/0.2165, 0.625/0.2165, 0.5/0.433}{
						\draw[color = red, fill = white] (\a + \dx, \b) circle (0.02);
					}
				}
				\draw[line width = 0.4mm, double, double distance = 1.5mm, color = white, -implies] (0.5, -0.02) -- (0.5, -0.25);
			\end{scope}
		\end{tikzpicture}
	}
	\caption{Complex containing the Hu--Zhang elements~\cite{Hu2014} at lowest order.}\label{fig:hu-zhang}
\end{figure}
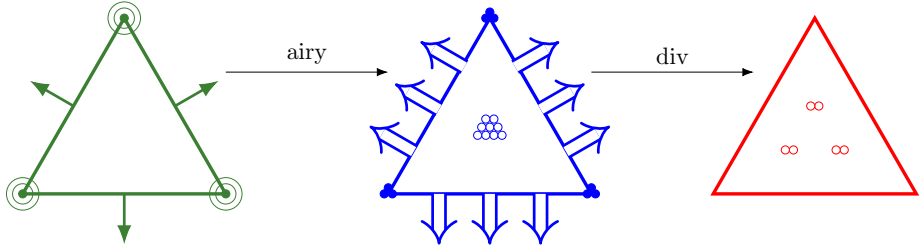

\subsection{Global degrees of freedom for the stress space}

\begin{lemma}\label{lem:hdiv-global-dofs}
    \emph{(Unisolvency for the global tensor space.)}
    The space $\hdivspace^{p}$, $p \geq 3$, has dimension
    \begin{align}\label{eq:dim-hdiv}
        \dim \hdivspace^{p} = 3|\mathcal{V}| + 2(p - 1) |\mathcal{E}| + \frac{3}{2} p(p - 1) |\mathcal{T}|
    \end{align}
    and is unisolvent with respect to the following DOFs for $\sigma \in \hdivspace^{p}$:
    \begin{enumerate}[label = (\roman*)]
        \item 3 values per vertex $\bx\in \mathcal{V}$: $\sigma(\bx)$,
        \item $2(p - 1)$ values per edge $\gamma \in \mathcal{E}$:
        \begin{equation}
            \int_{\gamma} (\sigma \bn) \cdot r~\ds \qquad \forall~r \in \mathcal{P}_{p - 2}(\gamma; \mathbb{V}),
        \end{equation}
        \item $\frac{3}{2} p(p - 1)$ values per element $K \in \mathcal{T}$:
        \begin{equation*}
            \int_{K} \sigma:\tau~\dx \qquad \forall~\tau \in \hdivspace_{0}^p(K; \mathbb{S}).
        \end{equation*}
    \end{enumerate}
\end{lemma}
\begin{proof}
    Suppose all global DOFs vanish on some $\sigma\in\hdivspace^p$. Then all local DOFs on each element $K\in\mathcal{T}$ vanish, so $\sigma\vert_K = 0$ on each $K$ by~\Cref{lem:hdiv-local-dofs}, and so $\sigma = 0$ in $\Omega$.
    Thus
    \begin{equation}\label{eq:global-unisolvency-ineq}
        \dim\hdivspace^p \leq 3|\mathcal{V}| + 2(p - 1) |\mathcal{E}| + \frac{3}{2} p(p - 1) |\mathcal{T}|.
    \end{equation}
    Now assume that values for each of the global DOFs are given. For each $K\in\mathcal{T}$, define $\sigma_K\in\mathcal{P}_p(K; \mathbb{S})$ by assigning the given values to the local DOFs. Then $\sigma$ defined by $\sigma\vert_K = \sigma_K$ is continuous at all mesh vertices, and has continuous normal component across cell edges, so $\sigma\in\hdivspace^p$.
    Thus~\cref{eq:global-unisolvency-ineq} holds as equality.
\end{proof}

We will also use an alternative set of degrees of freedom which characterize the space $\hdivspace_{\Gamma}^{p}$ in terms of moments of functions in $\htwospace_{\Gamma}^{p+2}$ and $\ltwospace_{\Gamma}^{p-1}$ and in terms of average fluxes around holes. Then we have the following result:
\begin{lemma}\label{lem:sigma-projection-dofs}
	The space $\hdivspace_{\Gamma}^{p}$, $p \geq 3$, has dimension
	\begin{equation*}
		\dim \hdivspace_{\Gamma}^p = \dim \htwospace_{\Gamma}^{p+2} + \dim \ltwospace_{\Gamma}^{p-1} + 3|\mathfrak{I}^*| - \dim \mathcal{P}_{1, \Gamma}(\Omega)
	\end{equation*}
	and is unisolvent with respect to the following DOFs for $\sigma \in \hdivspace_{\Gamma}^{p}$:
	\begin{enumerate}[label = (\roman*)]
		\item $\dim \airy \htwospace_{\Gamma}^{p+2} = \dim \htwospace_{\Gamma}^{p+2} - \dim \mathcal{P}_{1, \Gamma}(\Omega) $ values: $(\sigma, \airy q)$ for all $q \in \htwospace_{\Gamma}^{p+2}$,
		\item $\dim \ltwospace_{\Gamma}^{p-1}$ values: $(\dive \sigma, v)$ for all $v \in \ltwospace_{\Gamma}^{p - 1}$,
		\item 3 values for each element in $\mathfrak{I}^*$: $\langle \sigma \bn, r \rangle_{\partial \Omega_m}$ for all $r \in \rigidbodies$ and $m \in \mathfrak{I}^*$.
	\end{enumerate}
\end{lemma}
In other words, for each connected component of the boundary of $\Omega$ which does not lie entirely in $\Gamma_D$ (apart from one), there are 3 
additional linearly independent
divergence-free functions in $\hdivspace_\Gamma^p$ which 
do not lie in the space $\airy \htwospace_{\Gamma}^{p+2}$.
\Cref{lem:sigma-projection-dofs} will be proved in~\cref{sec:proof-projection-dofs}.

\section{Summary of main results}\label{sec:main-results}

In this section, we summarize the three main results: (i) a right inverse of the divergence operator that is bounded uniformly in $h$ and $p$; (ii) $hp$-bounded commuting cochain projections; and (iii) $hp$-stable Hodge decompositions. We then show that these results apply to the Arnold--Winther element~\cite{Arnold2002}.

\subsection{Uniformly stable inversion of the divergence operator and consequences}

The first result shows that the divergence operator from $\hdivspace_{\Gamma}^{p}$ to $\ltwospace_{\Gamma}^{p - 1}$ admits a right inverse which is bounded uniformly in $h$ and $p$.
\begin{theorem}\label{thm:invert-div-fem}
    Let $p \geq 3$. For every $u \in \ltwospace_{\Gamma}^{p-1}$, there exists $\sigma \in \hdivspace_{\Gamma}^{p}$ such that
    \begin{equation}\label{eq:global-div-interp}
        \dive \sigma = u \quad \text{and} \quad \|\sigma\|_{\dive, \Omega} \leq C \|u\|_{0, \Omega},
    \end{equation}
    where $C$ is independent of 
    $h$ and $p$. Consequently, the pair $\hdivspace_{\Gamma}^{p} \times \ltwospace_{\Gamma}^{p-1}$ is inf-sup stable uniformly in $h$ and $p$:
    \begin{equation}\label{eq:inf-sup-def}
        \beta(h, p) := 
        \inf_{u\in\ltwospace_{\Gamma}^{p-1}} \sup_{\sigma\in\hdivspace_{\Gamma}^{p}} \frac{ (\dive \sigma, u) }{ \|\sigma\|_{\dive, \Omega} \|u\|_{0, \Omega}  } \geq C^{-1} > 0.
    \end{equation}
\end{theorem}
\Cref{thm:invert-div-fem} is the discrete analog of~\cref{thm:h1-inversion-div-bc}; its proof appears at the end of~\cref{sec:invert-div-mesh}. 
One application of the inf-sup condition~\cref{eq:inf-sup-def} is that the discretization of the Hellinger--Reissner formulation~\cref{eq:hellinger-reissner-fem} with $\hdivspace^p_{\Gamma}$ and $\ltwospace^{p - 1}_{\Gamma}$ chosen as in~\cref{eq:fem-spaces-bcs} -- 
as well as any other finite element formulation requiring the same inf-sup condition on $\hdivspace^p_\Gamma\times\ltwospace^{p - 1}_\Gamma$, such as the linear Reissner--Mindlin plate considered in~\cite{Sky2023} --
is uniformly stable in $h$ and $p$. The right inverse of the divergence operator in~\cref{thm:invert-div-fem} also plays a crucial role in the remaining results of this section.

\subsection{$hp$-bounded commuting cochain projections}\label{sec:commuting-projections}

In this section, we define bounded projection operators so that the following diagram commutes:
\begin{equation}\label{eq:commuting-diagram}
    \begin{tikzcd}[row sep=large]
        0 \arrow{r}{} & \mathcal{P}_{1, \Gamma}(\Omega) \arrow{r}{\iota} \arrow{d}{I} &  H_{\Gamma}^{2}(\Omega) \arrow{r}{\airy} \arrow{d}{ \Pi_{\htwospace}^{p+2} }  &  \hdivgamma{\Omega}{\mathbb{S}} \arrow{r}{\dive} \arrow{d}{ \Pi_{\hdivspace}^p } & L_{\Gamma}^{2}(\Omega; \mathbb{V}) \arrow{r} \arrow{d}{ \Pi_{\ltwospace}^{p-1} } & 0 \\
        0 \arrow{r} & \mathcal{P}_{1,\Gamma}(\Omega) \arrow{r}{\iota} & \htwospace_{\Gamma}^{p+2} \arrow{r}{\airy} & \hdivspace_{\Gamma}^{p} \arrow{r}{\dive}  & \ltwospace_{\Gamma}^{p-1} \arrow{r} & 0.
    \end{tikzcd}
\end{equation}
As summarized 
in~\cite[Chapter 5]{Arnold2018}
and~\cite[Section 3]{Arnold2010}, 
the existence of these operators allows one to estimate the gap between continuous and discrete harmonic forms and to obtain discrete Poincar\'{e} inequalities, inf-sup conditions, and quasi-optimal error estimates for mixed problems. Additionally, the operators are also useful in discrete compactness results~\cite{He2019}. 
In the context of $h$-version finite elements, these cochain projection operators are typically locally defined and used to construct a right inverse of the divergence operator with the properties stated in~\cref{thm:invert-div-fem}. Here, we do the reverse: we use the right inverse of the divergence operator in~\cref{thm:invert-div-fem} to construct global cochain projections.

The operators $\Pi_Q^p : H^2_{\Gamma}(\Omega) \to \htwospace_{\Gamma}^{p}$, $p \geq 1$, and $\Pi_{\ltwospace}^{p} : L^2_{\Gamma}(\Omega) \to \ltwospace_{\Gamma}^{p}$, $p \geq 0$, are taken to be $H^2$- and $L^2$-projections:
\begin{align}\label{eq:piq-def}
    \left\{ 
    \begin{aligned}
        (\airy \Pi_Q^p q , \airy r) &= (\airy q, \airy r) \qquad & &\forall~r \in \htwospace_{\Gamma}^{p}, \\
        (\Pi_Q^p q, \ell) &= (q, \ell) \qquad & &\forall~\ell \in \mathcal{P}_{1, \Gamma}(\Omega),
    \end{aligned}
    \right.
    \qquad \forall~q \in H^2_{\Gamma}(\Omega),
\end{align}
and
\begin{equation}\label{eq:piv-def}
    (\Pi_V^{p} v, u) = (v, u) \qquad \forall~u \in \ltwospace_{\Gamma}^{p}, \ \forall~v \in  L^2_{\Gamma}(\Omega).
\end{equation}
We define the operator $\Pi_{\hdivspace}^p : \hdivgamma{\Omega}{\mathbb{S}} \to \hdivspace_{\Gamma}^p$, $p \geq 3$, via the degrees of freedom in~\cref{lem:sigma-projection-dofs}:
\begin{align}\label{eq:pi-sigma-def}
    \left\{ 
    \begin{aligned}
    (\Pi_{\hdivspace}^p \sigma, \airy q) &= (\sigma, \airy q) \qquad & &\forall~q \in \htwospace_{\Gamma}^{p+2}, \\
    (\dive \Pi_{\hdivspace}^p \sigma , v) &= (\dive \sigma, v) \qquad & &\forall~v \in \ltwospace_{\Gamma}^{p-1}, \\
    \langle (\Pi_{\hdivspace}^{p} \sigma) \bn, r \rangle_{\partial \Omega_m} &= \langle \sigma \bn, r \rangle_{\partial \Omega_m}, \qquad & &\forall~r \in \rigidbodies, \ \forall~m \in \mathfrak{I}^*,
    \end{aligned}
     \right. \quad \forall~\sigma \in \hdivgamma{\Omega}{\mathbb{S}}.
\end{align}
The operators are clearly well-defined linear operators. The following result shows that they are bounded uniformly in $h$ and $p$ and that the diagram~\cref{eq:commuting-diagram} commutes:
\begin{theorem}\label{thm:commuting-diagram}
    Let $p \geq 3$. The operators $\Pi_{\htwospace}^{p+2}$, $\Pi_{\hdivspace}^{p}$, and $\Pi_{\ltwospace}^{p-1}$ are projections and are bounded uniformly in $h$ and $p$ in the sense that
    \begin{subequations}
        \begin{alignat}{2}
            \label{eq:pi-q-bounded}
            \| \Pi_{\htwospace}^{p+2} q \|_{2, \Omega} &\leq C \| q\|_{2, \Omega} \qquad & &\forall~q \in H^2_{\Gamma}(\Omega), \\
            \label{eq:pi-sigma-bounded}
            \| \Pi_{\hdivspace}^{p} \sigma\|_{\dive, \Omega} &\leq C \|\sigma\|_{\dive, \Omega} \qquad & &\forall~\sigma \in \hdivgamma{\Omega}{\mathbb{S}}, \\
            \label{eq:pi-v-bounded}
            \| \Pi_{\ltwospace}^{p-1} v \|_{0, \Omega} &\leq \|v\|_{0, \Omega} \qquad & &\forall~v \in L^2_{\Gamma}(\Omega; \mathbb{V}),
        \end{alignat}
    \end{subequations}
    where $C$ is independent of $h$ and $p$. Moreover, the diagram~\cref{eq:commuting-diagram} commutes, i.e. 
    \begin{subequations}
        \begin{alignat}{2}
            \label{eq:commute-airy}
            \Pi_{\hdivspace}^{p} \airy q &= \airy \Pi_{\htwospace}^{p+2} q \qquad & &\forall~q \in H^2_{\Gamma}(\Omega), \\
            \label{eq:commute-div}
            \Pi_{\ltwospace}^{p-1} \dive \sigma &= \dive \Pi_{\hdivspace}^{p} \sigma \qquad & &\forall~\sigma \in \hdivgamma{\Omega}{\mathbb{S}}.
        \end{alignat}
    \end{subequations}
\end{theorem}
The proof of~\cref{thm:commuting-diagram} appears in~\cref{sec:proof-commuting-diagram}. Constructions of local bounded commuting operators may be possible; see e.g.~\cite{Chaumont-Frelet2023, Ern2022} for vector-valued spaces.

\subsection{$hp$-stable Hodge decompositions}\label{sec:stable-decomp}

Hodge decompositions play a key role in the well-posedness of Hodge-Laplace problems (see e.g.~\cite[\S4.4.1]{Arnold2018}) and in the construction of preconditioners for these problems~\cite{Arnold2000, Hiptmair1998, Lee2007}. The complex~\cref{eq:exact-sequence-fem} immediately gives rise to a Hodge decomposition. Let 
\begin{align}
    \label{eq:div-free-def}
    N_{\Gamma}^{p} &:= \{ \rho \in \hdivspace_{\Gamma}^{p} : \dive \rho \equiv 0 \}, \\
    \label{eq:harmonic-forms-discrete-def}
    \mathfrak{H}_{\Gamma}^p &:= \{ \phi \in N_{\Gamma}^{p} : (\phi, \airy q)_{\div, \Omega} = 0 \ \forall~q \in \htwospace_{\Gamma}^{p+2} \}, \\
    \label{eq:div-free-orthog-def}
    (N_{\Gamma}^{p})^{\perp} &:= \{ \psi \in \hdivspace_{\Gamma}^{p} : (\psi, \rho)_{\div, \Omega} = 0 \ \forall~\rho \in N_{\Gamma}^{p} \},
\end{align}
where~\cref{eq:harmonic-forms-discrete-def} is the space of discrete harmonic forms. Then standard arguments (see e.g.~\cite[Theorem 4.5]{Arnold2018}) show that for every $\sigma \in \hdivspace_{\Gamma}^{p}$, there exists $\phi \in \mathfrak{H}_{\Gamma}^p$, $q \in \htwospace_{\Gamma}^{p+2}$, and $\tau \in (N_{\Gamma}^{p})^{\perp}$ satisfying
\begin{equation*}
    \sigma = \phi + \airy q + \tau \quad \text{and} \quad \|\phi\|_{\dive, \Omega}^2 + \|  \airy q \|_{\dive, \Omega}^2 + \| \tau \|_{\dive, \Omega}^2 = \| \sigma \|_{\dive, \Omega}^2.    
\end{equation*}
Note that the above decomposition is also orthogonal with respect to the $L^2(\Omega; \mathbb{S})$-inner product. We further have the following discrete analog of~\cref{lem:exact-sequence-continuous-disp}:
\begin{lemma}\label{lem:dim-harmonic-discrete}
    For $p \geq 3$, there holds
    \begin{align*}
        \dim \mathfrak{H}_{\Gamma}^p = \dim \mathfrak{H}_{\Gamma}(\Omega) = 3|\mathfrak{I}^*|,
    \end{align*}
    where $\mathfrak{H}_{\Gamma}(\Omega)$ is the space of continuous harmonic forms defined in~\cref{eq:harmonic-forms-def}. Consequently, if $|\mathfrak{I}^*| = 0$, then the discrete sequence~\cref{eq:exact-sequence-fem} is exact.
\end{lemma}
\begin{proof}
	The result follows from~\Cref{lem:sigma-projection-dofs,lem:exact-sequence-continuous-disp} on noting that $\phi \in \mathfrak{H}_{\Gamma}^p$ if and only if the DOFs in (i) and (ii) in the statement of~\Cref{lem:sigma-projection-dofs} vanish. 
\end{proof}

Additionally, the following result, whose proof appears in~\cref{sec:proof-stable-decomp}, shows that, at the expense of orthogonality, one can fix the degree of the $\phi$ component to be the lowest order $(p = 3)$:
\begin{theorem}\label{thm:stable-decomposition-low-airy-div}
    Let $p \geq 3$. For every $\sigma \in \hdivspace_{\Gamma}^{p}$, there exists $\phi_3 \in \mathfrak{H}_{\Gamma}^3$, $q \in \htwospace_{\Gamma}^{p+2}$, and $\tau \in (N_{\Gamma}^{p})^{\perp}$ satisfying $\sigma = \phi_3 + \airy q + \tau$, 
    \begin{equation}
        \label{eq:low-order-div-free-high-order-decomp}
        \| \phi_3\|_{\dive, \Omega} + \|q\|_{2, \Omega} \leq C \|\sigma\|_{\dive, \Omega}, \quad \text{and} \quad
        \|\tau\|_{\dive, \Omega} \leq C \|\dive \sigma\|_{0, \Omega},
    \end{equation}
    where $C$ is independent of $h$ and $p$.
\end{theorem}
One consequence of~\cref{lem:dim-harmonic-discrete,thm:stable-decomposition-low-airy-div} 
of independent theoretical interest
is that if one views the complex~\cref{eq:exact-sequence-fem} as a scale of complexes~\cite{Arnold2021} indexed by the polynomial degree $p$, then the space $\mathfrak{H}_{\Gamma}^3$, which consists of piecewise-cubic divergence-free symmetric tensors orthogonal to $\airy \htwospace_{\Gamma}^{5}$, \textit{uniformly represents the cohomology} of the scale of complexes~\cref{eq:exact-sequence-fem} in the sense of~\cite{Arnold2021}. In other words, the space $\mathfrak{H}_{\Gamma}^3$ captures the structure of both the discrete cohomology space $\mathfrak{H}_{\Gamma}^p$ for all $p \geq 3$ and the continuous cohomology space $\mathfrak{H}_{\Gamma}(\Omega)$.

 
\subsection{Extension to the Arnold-Winther space}

One alternative to the stress space $\hdivspace_{\Gamma}^{p}$ is the more well-known Arnold--Winther symmetric stress space~\cite{Arnold2002} given (in terms of $\hdivspace^p$) by
 \begin{equation}\label{eq:arnold-winther-def}
    \tilde{\hdivspace}^p := \{\sigma\in \hdivspace^{p} :\Div\sigma\in \ltwospace^{p - 2}\} \quad \text{and} \quad \tilde{\hdivspace}_{\Gamma}^p := \tilde{\hdivspace}^p \cap \hdivgamma{\Omega}{\mathbb{S}}.
 \end{equation}
This space satisfies the same sequence property~\cref{eq:exact-sequence-fem} with $\hdivspace_{\Gamma}^{p}$ and $\ltwospace_{\Gamma}^{p-1}$ replaced by $\tilde{\hdivspace}^p$ and $\ltwospace_{\Gamma}^{p-2}$. In fact, almost every result in this work with $\hdivspace_{\Gamma}^{p}$ and $\ltwospace_{\Gamma}^{p-1}$ replaced by $\tilde{\hdivspace}^p_\Gamma$ and $\ltwospace_{\Gamma}^{p-2}$ holds with the exact same proof, including all of the results in this section. For concreteness, we state the analog of~\Cref{thm:invert-div-fem} and~\Cref{thm:stable-decomposition-low-airy-div}.
 \begin{corollary}
    \emph{(The Arnold--Winther space.)}
    Let $p\geq 3$. For every $u \in \ltwospace_{\Gamma}^{p - 2}$, there exists $\sigma \in \tilde{\hdivspace}_{\Gamma}^{p}$ satisfying~\cref{eq:global-div-interp} with the same constant. 
    Moreover, for every $\sigma \in \tilde{\Sigma}_{\Gamma}^{p}$, there exists $q \in Q_{\Gamma}^{p + 2}$,
 \begin{align*}
        \phi_3 &\in \{ \phi \in \tilde{\hdivspace}_{\Gamma}^3 : \dive \phi \equiv 0 \text{ and } (\phi, \airy r)_{\dive, \Omega} = 0 \ \forall~r \in Q_{\Gamma}^{p+2} \}, \text{ and } \\
        \tau &\in \{ \psi \in \tilde{\hdivspace}_{\Gamma}^{p} : (\psi, \rho)_{\dive, \Omega} = 0 \ \forall~\rho \in \tilde{\hdivspace}_{\Gamma}^{p}, \ \dive \rho \equiv 0 \}
 \end{align*}
    satisfying $\sigma = \phi_3 + \airy q + \tau$ and~\cref{eq:low-order-div-free-high-order-decomp} with $C$ independent of $h$ and $p$.
 \end{corollary}

 
\section{Numerical examples}\label{sec:numerics}

We now provide numerical demonstrations of uniform stability in the sense of~\Cref{thm:invert-div-fem}, before detailing the proof in the forthcoming sections.
We compute the inf-sup constant $\beta(h, p)$ in~\cref{eq:inf-sup-def}, which may be computed as the square root of the smallest eigenvalue $\nu$ of the generalized eigenproblem
 \begin{equation*}
    BA^{-1}B^\top p = \nu Cp,
 \end{equation*}
where $A_{ij} = (\tau_j, \tau_i)_{\Div, \Omega}$, $B_{ij} = (\Div\tau_j, v_i)_{\Omega}$, and $C_{ij} = (v_j, v_i)_{\Omega}$, for some choice of bases $\{\tau_i\}$ of $\hdivspace^p$ and $\{v_i\}$ of $\ltwospace^{p - 1}$~\cite[Section 3.5]{Elman2014}.
Using the finite element library Firedrake~\cite{Ham2023}\footnote{The Hu--Zhang elements are Piola-inequivalent, so that implementation of their bases in Firedrake required special transformations after mapping from the reference cell~\cite{Aznaran2022}.
Code, scripts, and exact software versions will be provided for reproducibility upon acceptance.
}, we plot these inf-sup constants for $3 \leq p \leq 12$ in~\cref{fig:inf-sup-p} for (i) the unit triangle $\{ (x, y) \in \mathbb{R}^2 : 0 \leq x, y , x + y \leq 1 \}$ with $|\Gamma| = |\Gamma_N|$ (full displacement), (ii) the unit triangle with $|\Gamma| = |\Gamma_D|$ (full traction), (iii) the unit square $(0, 1)^2$ meshed as in~\cref{fig:unit-square} with $|\Gamma| = |\Gamma_N|$, and (iv) a nonconvex and non-simply connected L-shape domain, with a nonuniform and asymmetric mesh as in~\cref{fig:L-shaped-domain}, with $|\Gamma| = |\Gamma_N|$. We also display in~\cref{fig:inf-sup-h} the inf-sup constants as a function of $h$ for uniform refinements (subdividing every triangle into four congruent subtriangles) of the unit square mesh. As predicted by~\Cref{thm:invert-div-fem}, the inf-sup constants are bounded away from zero as $p$ increases and/or as the mesh is refined. In fact, the inf-sup constants are near 1, which is the largest possible value since
\begin{align*}
    \beta(h, p) = \inf_{u\in\ltwospace_{\Gamma}^{p - 1}} \sup_{\sigma\in\hdivspace_{\Gamma}^{p}} \frac{ (\dive \sigma, u) }{ \|\sigma\|_{\dive, \Omega} \|u\|_{0, \Omega}  } \leq \inf_{u\in\ltwospace_{\Gamma}^{p - 1}} \sup_{\sigma\in\hdivspace_{\Gamma}^{p}} \frac{ \|\dive \sigma \|_{0,\Omega} \|u\|_{0, \Omega}  }{ \|\sigma\|_{\dive, \Omega} \|u\|_{0, \Omega}  } \leq 1.
\end{align*}

\begin{figure}[htb]
	\centering
	\begin{subfigure}[t]{0.4\linewidth}
		\includegraphics[width = 1.0\linewidth]{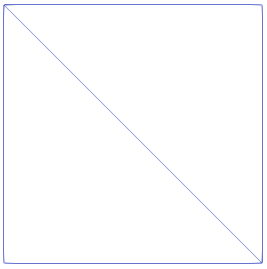}
		\caption{A coarse triangulation of $(0, 1)^2$.}\label{fig:unit-square}
	\end{subfigure}
	\hfill
	\begin{subfigure}[t]{0.4\linewidth}
		\includegraphics[width = 1.0\linewidth]{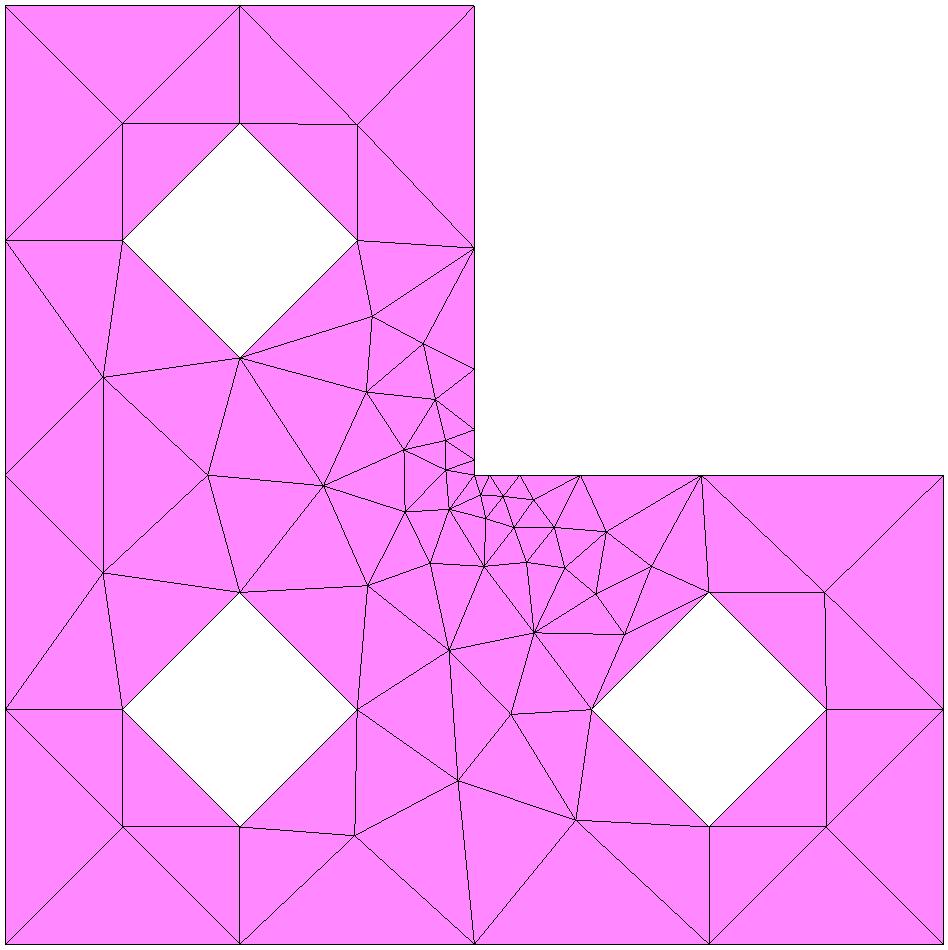}
		\caption{L-shape with holes, with a nonuniform, asymmetric mesh.}\label{fig:L-shaped-domain}
	\end{subfigure}
	\caption{Computational domains}
\end{figure}

 \begin{figure}[htb]
    \centering
    \begin{subfigure}[t]{0.48\linewidth}
        \includegraphics[width = 1.0\linewidth]{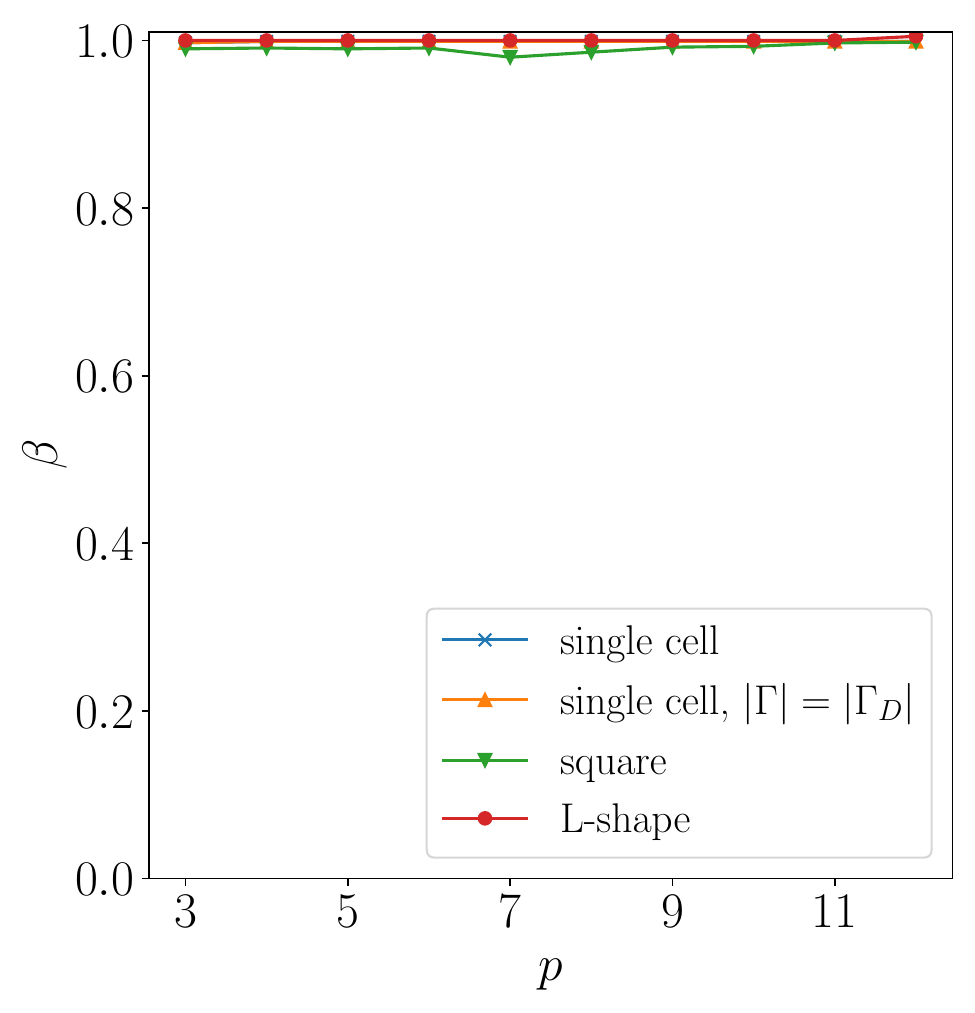}
        \caption{On the unit triangle (without and with traction conditions), on the mesh in~\cref{fig:unit-square}, and on the L-shape in~\Cref{fig:L-shaped-domain}.}\label{fig:inf-sup-p}
    \end{subfigure}
    \hfill
    \begin{subfigure}[t]{0.48\linewidth}
        \includegraphics[width = 1.0\linewidth]{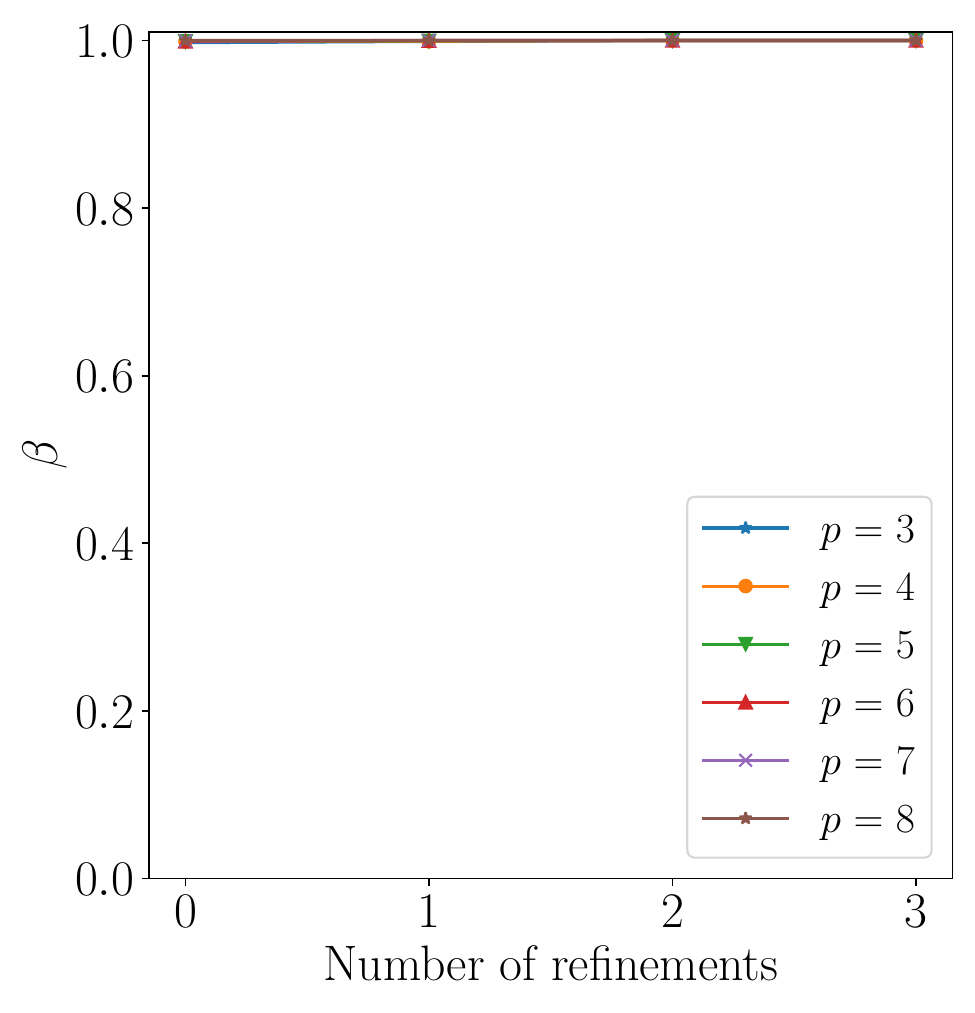}
        \caption{For uniform refinements of the mesh in~\Cref{fig:unit-square}.}\label{fig:inf-sup-h}
    \end{subfigure}
     \caption{Inf-sup constants $\beta(h, p)$ for the Hu--Zhang pair $\Sigma_\Gamma^p\times V_\Gamma^{p - 1}$ for various polynomial degrees $p$ and on various meshes.}
    \label{fig:infsup-computations}
 \end{figure} 

 
\section{Bounded Poincar\'{e} operators on an element}\label{sec:poincare-element}

 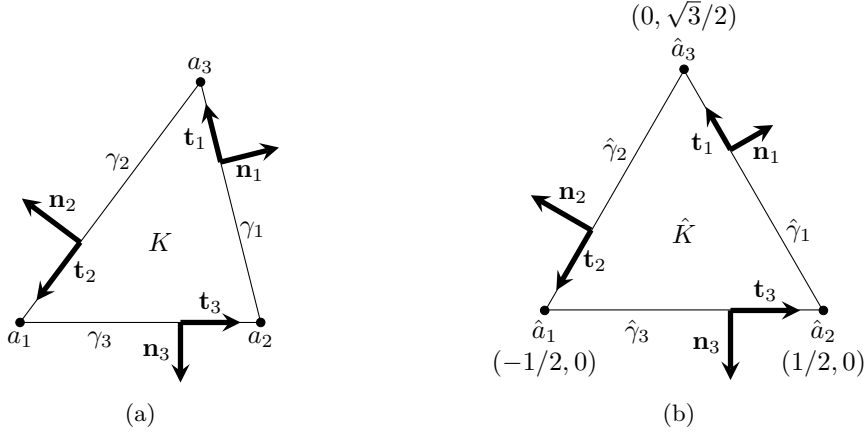
\begin{figure}[htb]
    \centering
    \begin{subfigure}[b]{0.45\linewidth}
        \centering
        \begin{tikzpicture}[scale=0.8]
            \filldraw (-2,0) circle (2pt) node[align=center,below]{$a_1$}
            -- (2,0) circle (2pt) node[align=center,below]{$a_2$}
            -- (1, 4) circle (2pt) node[align=center,above]{$a_3$}
            -- (-2,0);
            
            \coordinate (a1) at (-2,0);
            \coordinate (a2) at (2,0);
            \coordinate (a3) at (1, 4);
            
            \coordinate (e113) at ($(a2)!1/3!(a3)$);
            \coordinate (e123) at ($(a2)!2/3!(a3)$);
            \coordinate (e1231) at ($(a2)!2/3+1/sqrt(17)!(a3)$);
            
            \coordinate (e213) at ($(a3)!1/3!(a1)$);
            \coordinate (e223) at ($(a3)!2/3!(a1)$);
            \coordinate (e2231) at ($(a3)!2/3+1/sqrt(17)!(a1)$);
            
            \coordinate (e313) at ($(a1)!1/3!(a2)$);
            \coordinate (e323) at ($(a1)!2/3!(a2)$);
            \coordinate (e3231) at ($(a1)!2/3+1/4!(a2)$);
            
            \draw ($(e113)+(0.2,0.2)$) node[align=center]{$\gamma_1$};
            \draw ($(e213)+(0,0)$) node[align=center,left]{$\gamma_2$};
            \draw ($(e313)+(0,0)$) node[align=center,below]{$\gamma_3$};
            
            \draw[line width=2, -stealth] (e123) -- (e1231);
            \draw ($($(e123)!0.5!(e1231)$)+(-0.3,-0.1)$) node[align=center]{$\bt_1$};
            \draw[line width=2, -stealth] (e123) -- ($(e123)!1!-90:(e1231)$);
            \draw ($($(e123)!0.5!-90:(e1231)$)+(0.,-0.3)$) node[align=center]{$\bn_1$};
            \draw[line width=2, -stealth] (e223) -- (e2231);
            \draw ($($(e223)!0.5!(e2231)$)+(0.1,0)$) node[align=center, right]{$\bt_2$};
            \draw[line width=2, -stealth] (e223) -- ($(e223)!1!-90:(e2231)$);
            \draw ($($(e223)!0.5!-90:(e2231)$)+(0.2,0)$) node[align=center, above]{$\bn_2$};
            \draw[line width=2, -stealth] (e323) -- (e3231);
            \draw ($($(e323)!0.5!(e3231)$)+(0,0)$) node[align=center, above]{$\bt_3$};
            \draw[line width=2, -stealth] (e323) -- ($(e323)!1!-90:(e3231)$);
            \draw ($($(e323)!0.5!-90:(e3231)$)+(0,0)$) node[align=center, left]{$\bn_3$};
            
            \draw (1/3, 4/3) node(T){$K$};
            
        \end{tikzpicture}
        \caption{}
        \label{fig:general-triangle}
    \end{subfigure}
    \hfill
    \begin{subfigure}[b]{0.45\linewidth}
        \centering
        \begin{tikzpicture}[scale=0.8]
            
            \node[regular polygon, regular polygon sides=3, draw, minimum size=4.25cm] (m) at (0,0) {};
            
            \coordinate (a1) at (m.corner 2);
            \coordinate (a2) at (m.corner 3);
            \coordinate (a3) at (m.corner 1);
            
            \filldraw (a1) circle (2pt) node[align=center,below]{$\hat{a}_1$ \\ $(-1/2, 0)$};
            \filldraw (a2) circle (2pt) node[align=center,below]{$\hat{a}_2$ \\ $(1/2,0)$};
            \filldraw (a3) circle (2pt) node[align=center,above]{$(0, \sqrt{3}/2)$\\$\hat{a}_3$};

            \coordinate (e113) at ($(a2)!1/3!(a3)$);
            \coordinate (e123) at ($(a2)!2/3!(a3)$);
            \coordinate (e1231) at ($(a2)!2/3+1/sqrt(32)!(a3)$);
            
            \coordinate (e213) at ($(a3)!1/3!(a1)$);
            \coordinate (e223) at ($(a3)!2/3!(a1)$);
            \coordinate (e2231) at ($(a3)!2/3+1/4!(a1)$);
            
            \coordinate (e313) at ($(a1)!1/3!(a2)$);
            \coordinate (e323) at ($(a1)!2/3!(a2)$);
            \coordinate (e3231) at ($(a1)!2/3+1/4!(a2)$);
            
            \draw ($(e113)+(0,0)$) node[align=center,right]{$\hat{\gamma}_1$};
            \draw ($(e213)+(0,0)$) node[align=center,left]{$\hat{\gamma}_2$};
            \draw ($(e313)+(0,0)$) node[align=center,below]{$\hat{\gamma}_3$};
            
            \draw[line width=2, -stealth] (e123) -- (e1231);
            \draw ($($(e123)!0.5!(e1231)$)+(-0.25,-0.25)$) node[align=center]{$\bt_1$};
            \draw[line width=2, -stealth] (e123) -- ($(e123)!1!-90:(e1231)$);
            \draw ($($(e123)!0.5!-90:(e1231)$)+(0.25,-0.3)$) node[align=center]{$\bn_1$};
            \draw[line width=2, -stealth] (e223) -- (e2231);
            \draw ($($(e223)!0.5!(e2231)$)+(0,0)$) node[align=center, right]{$\bt_2$};
            \draw[line width=2, -stealth] (e223) -- ($(e223)!1!-90:(e2231)$);
            \draw ($($(e223)!0.5!-90:(e2231)$)+(0.2,0)$) node[align=center, above]{$\bn_2$};
            \draw[line width=2, -stealth] (e323) -- (e3231);
            \draw ($($(e323)!0.5!(e3231)$)+(0,0)$) node[align=center, above]{$\bt_3$};
            \draw[line width=2, -stealth] (e323) -- ($(e323)!1!-90:(e3231)$);
            \draw ($($(e323)!0.5!-90:(e3231)$)+(0,0)$) node[align=center, left]{$\bn_3$};

            \draw (m.center) node(T){$\hat{K}$};
        \end{tikzpicture}
        \caption{}
        \label{fig:reference-triangle}
    \end{subfigure}
    \caption{Notation for (a) general triangle $K$ and (b) reference triangle $\hat{K}$.}
 \end{figure}

We first consider the case that $\Omega = \hat{K}$ is the reference element~\cref{fig:reference-triangle}, pure traction conditions are imposed, and the mesh consists of exactly one element. In this case, the displacement belongs to the space
 \begin{equation*}
    \ltwomodrm(\hat{K}; \mathbb{V}) := \left\{ u \in L^2(\hat{K}; \mathbb{V}):(u, r)_{\Omega} = 0 \ \forall~r \in \rigidbodies \right\}. 
 \end{equation*}
The main result of this section is the following:
 \begin{theorem}[Inversion of the divergence with boundary conditions.]\label{thm:poincare-operators-element}
    There exist bounded linear operators
 \begin{equation*}
        P_1: \hdivz{\hat{K}}{\mathbb{S}} \to H^2_0(\hat{K}) \quad \text{and} \quad 
        P_2: \ltwomodrm(\hat{K}; \mathbb{V}) \to H^1(\hat{K}; \mathbb{S}) \cap \hdivz{\hat{K}}{\mathbb{S}} 
 \end{equation*}
    that satisfy a homotopy relation:
 \begin{equation}\label{eq:poincare-operators-element-homotopy}
        P_1 \airy = I, \quad \airy P_1 + P_2 \dive = I, \quad \text{and} \quad \dive P_2 = I.
 \end{equation}
    Moreover, $P_1$ maps $H^{s}(\hat{K}; \mathbb{S}) \cap \hdivz{\hat{K}}{\mathbb{S}}$ boundedly into $H^{s+2}(\hat{K}) \cap H^2_0(\hat{K})$ and $P_2$ maps $H^{s}(\hat{K}; \mathbb{V}) \cap \ltwomodrm(\hat{K}; \mathbb{V})$ boundedly into $H^{s+1}(\hat{K}; \mathbb{S}) \cap \hdivz{\hat{K}}{\mathbb{S}}$ for all real $s \geq 0$, with
 \begin{equation}\label{eq:poincare-operators-element-cont}
        \| P_1 \sigma\|_{s+2, \hat{K}} \leq C \{ \| \sigma \|_{s, \hat{K}} + \|\dive \sigma\|_{0, \hat{K}} \} \quad \text{and} \quad \|P_2 u \|_{s+1, \hat{K}} \leq C \|u\|_{s, \hat{K}},
    \end{equation}
    where $C$ depends only on $s$.
    
    Additionally, the operators preserve polynomials: for $p \geq 0$, there holds
 \begin{subequations}\label{eq:poincare-operators-element-poly-pres}
        \begin{align}
            P_1  (\mathcal{P}_{p}(\hat{K}; \mathbb{S}) \cap \hdivz{\hat{K}}{\mathbb{S}}) &\subseteq \mathcal{P}_{p+2}(\hat{K}) \cap H^2_0(\Omega), \\
            P_2( \mathcal{P}_{p}(\hat{K}; \mathbb{V}) \cap \ltwomodrm(\hat{K}; \mathbb{V})) &\subseteq \mathcal{P}_{p+1}(\hat{K}; \mathbb{S}) \cap \hdivz{\hat{K}}{\mathbb{S}}.
        \end{align}
 \end{subequations}
\end{theorem}
Note that the term $\|\dive \sigma\|_{0, \hat{K}}$ appears in~\cref{eq:poincare-operators-element-cont} to account for the cases $0 \leq s < 1$. The proof of~\cref{thm:poincare-operators-element} proceeds in several steps, detailed in the subsequent sections, and appears in~\cref{sec:proof-poincare-element}. The following corollary will be useful for global stability of the pair $\hdivspace^{p}_{\Gamma} \times \ltwospace^{p - 1}_{\Gamma}$:
 \begin{corollary}[Inversion of the divergence on a physical cell.]\label{cor:single-div-inversion}
    Let $p \geq 1$ and $K \in \mathcal{T}$. For every $u \in \mathcal{P}_{p-1}(K; \mathbb{V}) \cap \ltwomodrm(K; \mathbb{V})$, there exists $\sigma \in \mathcal{P}_{p}(K; \mathbb{S}) \cap \hdivz{K}{\mathbb{S}}$ such that
 \begin{equation}\label{eq:single-div-inversion}
        \dive \sigma = u \quad \text{and} \quad  h_K^{-1} \|\sigma\|_{0, K} + |\sigma|_{1, K} \leq C \|u\|_{0, K},
 \end{equation}
    where $C$ depends only on the shape regularity constant.
 \end{corollary}
\begin{proof}
    Let $u \in \mathcal{P}_{p-1}(K; \mathbb{V}) \cap \ltwomodrm(K; \mathbb{V})$ and let $F_K : \hat{K} \to K$ denote any invertible affine mapping with Jacobian $J_K = \nabla F_K$. Let $\hat{u} := J_K^{-1} u \circ F_K$ and $r \in \rigidbodies$. Then
 \begin{equation}
        (\hat{u}, r)_{\hat{K}} = (J_K^{-1} u \circ F_K, r)_{\hat{K}} = |\det J_K|^{-1} (u, J_K^{-\top} r \circ F_K^{-1})_{K} = 0,
 \end{equation}
    since $J^{-\top} r \circ F_K^{-1} \in \rigidbodies$. Thanks to~\cref{thm:poincare-operators-element}, there exists $\hat{\sigma} \in \mathcal{P}_{p-1}(\hat{K}; \mathbb{V}) \cap \ltwomodrm(\hat{K}; \mathbb{V})$ such that $\dive \hat{\sigma} = \hat{u}$ and
    \begin{align}
        \| \hat{\sigma} \|_{1, \hat{K}} \leq C \|\hat{u}\|_{0, \hat{K}} = C |\det J_K|^{-1} \| J_K^{-1} u \|_{0, K} \leq C h_K |\det J_K|^{-2} \|u\|_{0, K}.
    \end{align}
    Let $\sigma := J_K \hat{\sigma} J_K^\top \circ F_K^{-1}$. Thanks to the chain rule, $\dive \sigma = J_K \dive \hat{\sigma} \circ F_K^{-1} = u$, and
 \begin{equation*}
        \| \sigma \|_{0, K} = |\det J_K| \| J_K \hat{\sigma} J_K^\top \|_{0, \hat{K}} \leq 2 h_K^2 |\det J_K| \|\hat{\sigma}\|_{0, \hat{K}} \leq C h_K^3 |K|^{-1} \|u\|_{0, K}.
 \end{equation*}
    The $H^1(K; \mathbb{S})$ seminorm bound follows from a similar scaling argument.
\end{proof}

 
\subsection{Inverting divergence without boundary conditions}\label{sec:invert-div-no-bc}

The first step considers the operator $P_2$, which is an inverse of the divergence operator, and drops the constraints on the boundary traces present in~\cref{thm:poincare-operators-element}. We have the following result:
 \begin{lemma}[Inversion of the divergence without boundary conditions.]\label{lem:invert-div-ref-element-nobc}
    There exists a linear operator $\mathcal{L}_{\dive} : L^2(\hat{K}; \mathbb{V}) \to H^1(\hat{K}; \mathbb{S})$ satisfying the following: for any real number $s \geq 0$ and $u \in H^{s}(\hat{K}; \mathbb{V})$, $\mathcal{L}_{\dive} u \in H^{s+1}(\hat{K}; \mathbb{S})$ and
 \begin{equation}\label{eq:invert-div-ref-element-nobc}
        \dive \mathcal{L}_{\dive} u = u \quad \text{with} \quad \| \mathcal{L}_{\dive} u \|_{s+1, \hat{K}} \leq C \| u \|_{s, \hat{K}},
 \end{equation}
    where $C$ depends only on $s$. Additionally, if $u \in \mathcal{P}_{p}(\hat{K}; \mathbb{V})$, $p \geq 0$, then $\mathcal{L}_{\dive} u \in \mathcal{P}_{p+1}(\hat{K}; \mathbb{S})$.
 \end{lemma}
The proof of~\cref{lem:invert-div-ref-element-nobc} relies on the BGG construction in~\cite{Cap2023}, which begins with the following diagram:
 \begin{equation}
    \begin{tikzcd}[row sep = large, column sep = 0.5in]
        0 \arrow{r} &  H^{s+3}(\hat{K}) \arrow[r, "\D_0^A :=\curl", shift left]  &  H^{s+2}(\hat{K}; \mathbb{V}) \arrow[r, "\D_1^A :=\dive", shift left] \arrow[l, "P_1^{A}", shift left] &   H^{s+1}(\hat{K}) \arrow[r, "\D_2^A := 0", shift left] \arrow[l, "P_2^{A}", shift left] & 0 \arrow[l, "P_3^{A} := 0", shift left]\\
        0 \arrow{r} \arrow[ur, "S_{-1}", start anchor=north, end anchor=south] & H^{s+2}(\hat{K}; \mathbb{V}) \arrow[ur, "S_{0}", start anchor=north, end anchor=south] \arrow[r, "\D_0^B :=\curl", shift left] & H^{s+1}(\hat{K}; \mathbb{M}) \arrow[ur, "S_{1}", start anchor=north, end anchor=south] \arrow[r, "\D_1^B :=\dive", shift left] \arrow[l, "P_1^{B}", shift left] & H^{s}(\hat{K}; \mathbb{V}) \arrow[ur, "S_{2}", start anchor=north, end anchor=south] \arrow[r, "\D_2^B := 0", shift left] \arrow[l, "P_2^{B}", shift left] & 0 \arrow[l, "P_3^{B} := 0", shift left],
    \end{tikzcd}
 \end{equation}
where $\mathbb{M} := \mathbb{R}^{2\times 2}$, $s \geq 0$, and $\curl$ in the second row is applied row-wise. We equip each of the above spaces with its canonical inner product. The operators  $S_{-1} : \{ 0 \} \to H^{s+3}(\hat{K})$, $S_0 : H^{s+2}(\hat{K}; \mathbb{V}) \to H^{s+2}(\hat{K}; \mathbb{V})$, $S_1 :  H^{s+1}(\hat{K}; \mathbb{M}) \to H^{s+1}(\hat{K})$, and $S_2 :  H^{s}(\hat{K}; \mathbb{V}) \to \{0\}$ are defined as follows
 \begin{equation}\label{eq:bgg-s-operators}
    S_{-1} := \iota, \quad S_0 := I, \quad S_1 = -2\sskw, \quad \text{and} \quad S_2 := 0,
 \end{equation}
where $\sskw m := \frac{1}{2}(m_{12} - m_{21})$. The diagram satisfies an anticommutativity property:
 \begin{equation}\label{eq:bgg-s-operators-anitcommute}
    \curl \circ S_{-1} = - S_0 \circ \iota, \quad \dive \circ S_0 = -S_1 \circ \curl, \quad \text{and} \quad 0 \circ S_1 = - S_2 \circ \dive. 
 \end{equation}

The operators $P_i^A$ and $P_i^B$ are related to the bounded (regularized) Poincar\'{e} operators from~\cite{Costabel2010}. Let $\mathcal{C}_{\curl} : H^{s}(\hat{K}; \mathbb{V}) \to H^{s+1}(\hat{K})$ and $\mathcal{C}_{\dive} : H^{s}(\hat{K}) \to H^{s+1}(\hat{K}; \mathbb{V})$, $s \geq 0$, defined by 
 \begin{align}
    \label{eq:costabel-poincare-curl}
    \mathcal{C}_{\curl}(v)(x) &= \int_{\hat{K}} \theta(y) (x - y) \cdot \int_{0}^{1} v^{\perp}(y + t(x - y)) \ \dt \ \dy, \\
    \label{eq:costabel-poincare-div}
    \mathcal{C}_{\dive}(q)(x) &= -\int_{\hat{K}} \theta(y) (x - y)\int_{0}^{1} t q(y + t(x - y)) \ \dt \ \dy,
 \end{align}
where we recall that $v^{\perp} := (-v_2, v_1)^{\top}$, and $\theta \in C^{\infty}_c(\hat{K})$ is arbitrary but fixed with $\int_{\hat{K}} \theta(y) \ \dy = 1$. These operators satisfy
 \begin{equation*}
    \curl \mathcal{C}_{\curl} + \mathcal{C}_{\dive} \div = I \quad \text{and} \quad \dive \mathcal{C}_{\dive} = I,
 \end{equation*}
and preserve polynomials in the sense that 
 \begin{align}
    \label{eq:costabel-poly-preserving}
    \mathcal{C}_{\curl} \mathcal{P}_{p}(\hat{K}; \mathbb{V}) \subset \mathcal{P}_{p+1}(\hat{K}) \quad \text{and} \quad \mathcal{C}_{\dive} \mathcal{P}_{p}(\hat{K}) \subset \mathcal{P}_{p+1}(\hat{K}; \mathbb{V}).
 \end{align}
We then define $P_1^A := \mathcal{C}_{\curl}$,  $P_2^A := \mathcal{C}_{\dive}$, and $P_i^B$ is $P_i^A$ applied row-wise, $i \in \{1,2\}$.

To proceed further, we define (pseudo-)inverses of the operators $S_i$:
\begin{align}
    \label{eq:bgg-t-operators}
    T_{-1} := 0, \quad T_0 := I, \quad T_1 = -\frac{1}{2}\mskw, \quad \text{and} \quad T_2 := \iota,
\end{align}
where $\mskw u:= \begin{pmatrix}
    0 & u\\
    -u & 0\end{pmatrix}$.
We will use the notation $\ran$ and $\ker$ to denote the range and nullspace of an operator, respectively. Additionally, given a subspace $G \subset H$, 
let $\Pi_{G} : H \to G$ denote the orthogonal projection onto $G$. Some useful properties of $S_i$ and $T_i$ are recorded in the following lemma:
 \begin{lemma}\label{lem:bgg-s-operators-properties}
    The map
    $S_{-1}$ is injective, $S_{0}$ is bijective, and $S_1$ and $S_2$ are surjective. Moreover, $S_{i}$ and $T_i$, $-1 \leq i \leq 2$, are bounded with  
 \begin{equation}\label{eq:bgg-t-operators-inverse-id}
        T_i = (S_i|_{\ker( S_i)^{\perp}} )^{-1} \circ \Pi_{\ran (S_i)},
 \end{equation}
    where $ (S_i|_{\ker(S_i)^{\perp}} )^{-1}$ denotes the inverse of the restriction of $S_i$ to the orthogonal complement of $\ker(S_i)$.
    Finally, the operators are polynomial preserving for $p \geq 0$:
 \begin{alignat*}{2}
        S_{-1} \{0\} &\subset \mathcal{P}_{p}(\hat{K}), \qquad & T_{-1} \mathcal{P}_{p}(\hat{K}) &= \{0\}, \\ 
        S_0 \mathcal{P}_{p}(\hat{K}; \mathbb{V}) &=  \mathcal{P}_{p}(\hat{K}; \mathbb{V}), \qquad & T_0 \mathcal{P}_{p}(\hat{K}; \mathbb{V}) &=  \mathcal{P}_{p}(\hat{K}; \mathbb{V}), \\
        S_1 \mathcal{P}_{p}(\hat{K}; \mathbb{M}) &= \mathcal{P}_{p}(\hat{K}), \qquad & T_1 \mathcal{P}_{p}(\hat{K}) &\subset \mathcal{P}_{p}(\hat{K}; \mathbb{M}),  \\
        S_2 \mathcal{P}_p(\hat{K}; \mathbb{V}) &= \{0\}, \qquad & T_2\{0\}&\subset \mathcal{P}_p(\hat{K}; \mathbb{V}).
 \end{alignat*}
 \end{lemma}
\begin{proof}
    Direct verification shows that $S_{-1}$ is injective, $S_{0}$ is bijective, and $S_1$ and $S_2$ are surjective and that $S_i$ and $T_i$, $-1 \leq i \leq 2$, are bounded. Identity~\cref{eq:bgg-t-operators-inverse-id} follows from definition for $i \neq 1$, while the case $i=1$ is an immediate consequence of the identity $\mathcal{N}(S_1)^{\perp} = \{ \mskw u : u \in H^{s+1}(\hat{K}) \}$.
    The polynomial preservation properties follow by inspection.
\end{proof}

 
\subsubsection{Twisted complex}

The next step in the construction is the \textit{twisted} complex:
 \begin{equation}\label{eq:twisted-complex}
    \begin{tikzcd}[ampersand replacement = \&, column sep = 0.925in]
        \begin{pmatrix}
            H^{s+3}(\hat{K}) \\
            H^{s+2}(\hat{K}; \mathbb{V}) 
        \end{pmatrix} \arrow{d}{F_{0}}
        \arrow{r}{ \begin{pmatrix}
                \curl & 0 \\
                0 & \curl
        \end{pmatrix}} 
        \& \begin{pmatrix}
            H^{s+2}(\hat{K}; \mathbb{V}) \\
            H^{s+1}(\hat{K}; \mathbb{M}) 
        \end{pmatrix} \arrow{d}{F_{1}}  \arrow{r}{ \begin{pmatrix}
                \div & 0 \\
                0 & \div
        \end{pmatrix}} 
        \&\begin{pmatrix}
            H^{s+1}(\hat{K}) \\
            H^{s}(\hat{K}; \mathbb{V}) 
        \end{pmatrix}  \arrow{d}{F_{2}}
        \\
        \begin{pmatrix}
            H^{s+3}(\hat{K}) \\
            H^{s+2}(\hat{K}; \mathbb{V}) 
        \end{pmatrix}  \arrow{r}{\D_0^{\text{twist}}} 
        \& \begin{pmatrix}
            H^{s+2}(\hat{K}; \mathbb{V}) \\
            H^{s+1}(\hat{K}; \mathbb{M}) 
        \end{pmatrix}   \arrow{r}{\D_1^{\text{twist}}}
        \&\begin{pmatrix}
            H^{s+1}(\hat{K}) \\
            H^{s}(\hat{K}; \mathbb{V}) 
        \end{pmatrix},
    \end{tikzcd}
 \end{equation}
where
 \begin{equation*}
    \D_i^{\text{twist}} := \begin{pmatrix}
        \D_i^A & -S_i \\
        0 & \D_i^B
    \end{pmatrix} \quad \text{and} \quad
    F_i := \begin{pmatrix}
        I & P_{i+1}^A \circ S_i \\
        0 & I
    \end{pmatrix}, \quad i \in \{0, 1, 2\}.
 \end{equation*}
The operators $F_i$ are clearly invertible, and so we define Poincar\'{e} operators for the twisted complex by the rule
 \begin{equation*}
    P_i^{\text{twist}} := F_{i-1} \circ \begin{pmatrix}
        P_i^A & 0 \\
        0 & P_i^B
    \end{pmatrix} \circ F_i^{-1}, \quad i \in \{1,2\}.
 \end{equation*}
In particular, the operators satisfy~\cite[Theorem 2]{Cap2023}
 \begin{equation*}
    \D_0^{\text{twist}} P_1^{\text{twist}} + P_2^{\text{twist}} \D_1^{\text{twist}} = I \quad \text{and} \quad \D_1^{\text{twist}} P_2^{\text{twist}} = I.
 \end{equation*}
Thanks to the continuity of the operators $S_j$, $j \in \{-1,0,1,2\}$, and $P_{i}^A$, $P_i^B$, $i \in \{1,2\}$, the operators  $P_i^{\text{twist}}$ are continuous.

 
\subsubsection{BGG complex}

The final step connects the twisted complex to the BGG complex:
 \begin{equation}\label{eq:bgg-complex}
    \begin{tikzcd}[ampersand replacement=\&, column sep=0.8in]
        \begin{pmatrix}
            H^{s+3}(\hat{K}) \\
            H^{s+2}(\hat{K}; \mathbb{V}) 
        \end{pmatrix} \arrow[d, "B_0", shift left]
        \arrow{r}{ \D_0^{\text{twist}} } 
        \& \begin{pmatrix}
            H^{s+2}(\hat{K}; \mathbb{V}) \\
            H^{s+1}(\hat{K}; \mathbb{M}) 
        \end{pmatrix} \arrow[d, "B_1", shift left]  \arrow{r}{ \D_1^{\text{twist}} } 
        \&\begin{pmatrix}
            H^{s+1}(\hat{K}) \\
            H^{s}(\hat{K}; \mathbb{V}) 
        \end{pmatrix}  \arrow[d, "B_2", shift left]
        \\
        \begin{pmatrix}
            \ran(S_{-1})^{\perp} \\
            \ker(S_0) 
        \end{pmatrix} \arrow[u, "A_0", shift left]  \arrow{r}{\D_0^{BGG} } 
        \& \begin{pmatrix}
            \ran(S_0)^{\perp} \\
            \ker(S_1) 
        \end{pmatrix} \arrow[u, "A_1", shift left]  \arrow{r}{\D_1^{BGG}}
        \&\begin{pmatrix}
            \ran(S_1)^{\perp} \\
            \ker(S_2) 
        \end{pmatrix} \arrow[u, "A_2", shift left],
    \end{tikzcd}
 \end{equation}
where
 \begin{equation*}
    A_i := \begin{pmatrix}
        I & 0 \\
        T_i \circ \D_i^A & \Pi_{\ker(S_i)}
    \end{pmatrix}, \quad 
    B_i := \begin{pmatrix}
        \Pi_{\ran(S_{i-1})^{\perp}} & 0 \\
        \Pi_{\ker(S_i)} \circ \D_{i-1}^B \circ T_{i-1} & \Pi_{\ker(S_i)} 
    \end{pmatrix}, 
 \end{equation*}
and
 \begin{equation*}
    \D_i^{BGG} := \begin{pmatrix}
        \Pi_{\ran(S_{i})^{\perp}} & 0 \\
        0 & \Pi_{\ker(S_{i+1})} 
    \end{pmatrix} \circ \D_i^{\text{twist}} \circ A_i.
 \end{equation*}
The operators $P_i^{BGG} := B_{i-1} \circ P_i^{\text{twist}} \circ A_i$ satisfy~\cite[Theorem 3]{Cap2023}
 \begin{equation}\label{eq:bgg-homotopy-property}
    \D_0^{BGG} P_1^{BGG} + P_2^{BGG} \D_1^{BGG} = I \quad \text{and} \quad \D_1^{BGG} P_2^{BGG} = I.
 \end{equation}
Additionally, $P_i^{BGG}$, $i \in \{1,2\}$ are bounded thanks to the boundedness of $A_i$, $B_i$ and $P_i^{\text{twist}}$. We are now in a position to prove~\cref{lem:invert-div-ref-element-nobc}.

 
\subsubsection{Proof of~\cref{lem:invert-div-ref-element-nobc}}

Let $\sym m = \frac{1}{2}(m + m^{\top})$ denote the symmetric part of a matrix $m$ and $\mathcal{L}_{\dive} := \sym \circ (\curl \circ P_2^A \circ S_1 + I) \circ P_2^B$. Then there holds
 \begin{align*}
    (\sym \circ \curl \circ P_2^A \circ S_1 \circ P_2^B) \mathcal{P}_{p}(\hat{K}; \mathbb{V}) &\subseteq (\sym \circ \curl \circ P_2^A \circ S_1 ) \mathcal{P}_{p+1}(\hat{K}; \mathbb{M})  \\
    &\subseteq (\sym \circ \curl \circ P_2^A) \mathcal{P}_{p+1}(\hat{K}) \\
    &\subseteq (\sym \circ \curl) \mathcal{P}_{p+2}(\hat{K}; \mathbb{V}) \\
    &\subseteq \sym \mathcal{P}_{p+1}(\hat{K}; \mathbb{M}),
 \end{align*}
and so $\mathcal{L}_{\dive}\mathcal{P}_{p}(\hat{K}; \mathbb{V}) \subseteq \mathcal{P}_{p+1}(\hat{K}; \mathbb{S})$, $p \geq 0$. 

We now show that~\cref{eq:invert-div-ref-element-nobc} follows from~\cref{eq:bgg-homotopy-property} and the boundedness of $P_2^{BGG}$. Since $S_0$ and $S_1$ are surjective, $\ran(S_0)^{\perp} = \ran(S_1)^{\perp} = 0$. Moreover, the identities $\mathcal{N}(S_1) = H^{s+1}(\hat{K}, \mathbb{S})$ and $\mathcal{N}(S_2) = H^{s}(\hat{K}; \mathbb{V})$ follow by inspection. Consequently,
 \begin{equation*}
    A_1 = \begin{pmatrix}
        I & 0 \\
        T_1 \circ \dive & \sym 
    \end{pmatrix}, \quad
    A_2 = \begin{pmatrix}
        I & 0 \\
        0 & I
    \end{pmatrix}, \quad \text{and} \quad B_1 = \begin{pmatrix}
        0 & 0 \\
        \sym \circ \curl & \sym
    \end{pmatrix}.
 \end{equation*}
Unpacking definitions then gives
 \begin{equation*}
    \D_1^{BGG} 
    = \begin{pmatrix}
        0 & 0 \\
        0 & \dive  
    \end{pmatrix} 
    \quad \text{and} \quad 
    P_2^{BGG} 
    = \begin{pmatrix}
        0 & 0 \\
        \sym \circ \curl \circ P_2^A & \mathcal{L}_{\dive}
    \end{pmatrix}.
 \end{equation*}
Consequently, the operator $\mathcal{L}_{\div}$ maps $H^s(\hat{K}; \mathbb{V})$ boundedly into $H^{s+1}(\hat{K}; \mathbb{S})$ since $P_2^{BGG}$ is bounded. The result now follows from~\cref{eq:bgg-homotopy-property}, which gives
 \begin{equation*}
    \begin{pmatrix}
        0 & 0 \\
        0 & \div \circ \mathcal{L}_{\div}
    \end{pmatrix} \begin{pmatrix}
        0 \\ u
    \end{pmatrix} = \begin{pmatrix}
        0 \\ u
    \end{pmatrix},
 \end{equation*}
\hfill\proofbox

 
\subsection{Inverting divergence with pure traction conditions}

We now turn to the construction of the operator $P_2$ in~\cref{thm:poincare-operators-element} with constraints on the boundary traces. The first result concerns the existence of a divergence-free extension of the boundary traces of symmetric tensors whose divergences belong to $\ltwomodrm(\hat{K}; \mathbb{V})$: 
a potential whose $\airy$ stress function has the same normal trace as the argument tensor field, but does not contribute to its divergence.
This will play the role of `correcting' the boundary values of the image of the Poincar\'e operator.
 \begin{lemma}[Divergence-free extension of boundary traces.]\label{lem:airy-correction-operator}
    For $s \geq 0$, let $Y^{s+1}(\hat{K}; \mathbb{S}) := \{ \sigma \in H^{s+1}(\hat{K}; \mathbb{S}) : \dive \sigma \in \ltwomodrm(\hat{K}; \mathbb{V}) \}$. Then there exists a linear operator $\mathcal{L}_{\bn} : Y^1(\hat{K}; \mathbb{S}) \to H^{3}(\hat{K})$ that maps $Y^{s+1}(\hat{K}; \mathbb{S})$ boundedly into $H^{s+3}(\hat{K})$ for all real $s \geq 0$ and satisfies
 \begin{align}\label{eq:airy-correction-properties}
        (\airy \mathcal{L}_{\bn} - I) \tau \bn|_{\partial \hat{K}} = 0 \quad \text{and} \quad \|  \mathcal{L}_{\bn} \tau \|_{s+3, \hat{K}} \leq C \|\tau\|_{s+1, \hat{K}} \qquad \forall~\tau \in Y^{s+1}(\hat{K}; \mathbb{S}),
 \end{align}
    where $C$ depends only on $s$. Moreover, if $\tau \in \mathcal{P}_{p}(\hat{K}; \mathbb{S})$, $p \geq 0$, with $\dive \tau \in \ltwomodrm(\hat{K}; \mathbb{V})$, then  $\mathcal{L}_{\bn} \tau \in \mathcal{P}_{p+2}(\hat{K})$ and~\cref{eq:airy-correction-properties} holds for all $s \geq 0$.
 \end{lemma}
\begin{proof}
    We recall two key results concerning the traces of $H^{s+3}(\hat{K})$ functions. First, let $X^{s+3}(\partial \hat{K})$ denote the trace space
 \begin{equation}
        X^{s+3}(\partial \hat{K}) := \{ (q|_{\partial \hat{K}}, \partial_{\bn} q|_{\partial \hat{K}}) : q \in H^{s+3}(\hat{K}) \}, \qquad s \geq 0.
 \end{equation}
    Theorem 6.1 of~\cite{Arnold1988} states that $(f, g) \in X^{s+3}(\partial \hat{K})$ if and only if:
    \begin{enumerate}
        \item $f|_{\hat{\gamma}_i} \in H^{s+5/2}(\hat{\gamma}_i)$ and $ g|_{\hat{\gamma}_i} \in H^{s+3/2}(\hat{\gamma}_i)$, $i \in \{1,2,3\}$.
        \item $f$ and $\mathfrak{D}(f, g) := \partial_{\bt} f \bt + g \bn$ are continuous.
        \item For $i \in \{1,2,3\}$, there holds
 \begin{alignat}{2}
            &\mathcal{I}_i^2(\mathfrak{D}(f, g)) < \infty \quad & & \text{if } s = 0, \notag \\
            \label{eq:proof:traces-mixed-derivative}
            &\bt_{i+1} \cdot \partial_{\bt}\mathfrak{D}(f, g)|_{\hat{\gamma}_{i+2}}(\hat{a}_i) = \bt_{i+2} \cdot \partial_{\bt} \mathfrak{D}(f, g) |_{\hat{\gamma}_{i+2}}(\hat{a}_i) \qquad & & \text{if } s > 0,
 \end{alignat}
        where indices are understood modulo 3 and
 \begin{equation*}
            \mathcal{I}_i^2(v) := \int_{0}^{1} \frac{| \bt_{i+1} \cdot \partial_{\bt}v(\hat{a}_i + \delta \bt_{i+2}) - \bt_{i+2} \cdot \partial_{\bt} v(\hat{a}_i - \delta \bt_{i+1}) |^2}{\delta} \ \D \delta.
 \end{equation*}
    \end{enumerate}
    We equip the space $X^{s+3}(\partial \hat{K})$ with the norm
 \begin{equation*}
        \|(f, g)\|_{X^{s+3}(\partial \hat{K})}^2 := \sum_{i=1}^{3} \{ \|f\|_{s+5/2, \hat{\gamma}_i}^2 + \|g\|_{s+3/2, \hat{\gamma}_i}^2 \} + \sum_{i=1}^{3} \begin{cases}
            I^2_i(\mathfrak{D}(f, g)) & \text{if } s = 0, \\
            0 & \text{if } s \geq 0.
        \end{cases}
 \end{equation*}
    
    The second key result~\cite[Theorem 3.2]{Parker2023} shows that there exists an operator $\tilde{\mathcal{L}} : X^3(\partial \hat{K}) \to H^{3}(\hat{K})$ satisfying $\tilde{\mathcal{L}}(f, g) \in H^{s+3} (\hat{K})$ if $(f, g) \in X^{s+3}(\partial \hat{K})$ and
 \begin{equation}\label{eq:proof:hs3-lifting-operator}
        (\tilde{\mathcal{L}}(f, g)|_{\partial \hat{K}}, \partial_{\bn} \tilde{\mathcal{L}}(f, g)|_{\partial \hat{K}}) = (f, g) \quad \text{and} \quad \| \tilde{\mathcal{L}}(f, g) \|_{s+3, \hat{K}} \leq C \|(f, g)\|_{X^{s+3}(\partial \hat{K})},
 \end{equation}
    where $C$ depends only on $s$.
        
    Given $\tau \in Y^{s+1}(\hat{K}; \mathbb{S})$, $s \geq 0$, we construct functions $f$ and $g$ so that $(f, g) \in X^{s+3}(\partial \hat{K})$ and if $q \in H^{s+3}(\hat{K}$) with $(q, \partial_\bn q)|_{\partial \hat{K}} = (f, g)$, then $(\airy q) \bn = \tau \bn$ on $\partial \hat{K}$. To this end, let $v \in L^2(\partial \hat{K}; \mathbb{V})$ be defined by the rule
 \begin{equation*}
        v(x) = \int_{\hat{a}_1}^{x} \tau \bn~\ds, \qquad x \in \partial \hat{K},
 \end{equation*}
    where the path integral is taken counterclockwise around $\partial \hat{K}$. Since $\tau$ satisfies $\tau|_{\hat{\gamma}_i} \in H^{s+1/2}(\hat{\gamma}_i; \mathbb{V})$ and $( \tau \bn, \be_i )_{\partial \hat{K}} = (\dive \tau, \be_i)_{\hat{K}} = 0$ by the divergence theorem,  $v$ satisfies
 \begin{equation}\label{eq:proof:one-integral-trace-reg}
        v \text{ is continuous} \quad \text{and} \quad v|_{\hat{\gamma}_i} \in H^{s+3/2}(\hat{\gamma}_i; \mathbb{V}), \quad i \in \{1,2,3\}.
 \end{equation}
    
    Now define the functions $f, g : \partial \hat{K} \to \mathbb{R}$ by the rules
 \begin{equation*}
        f(x) = \int_{\hat{a}_1}^{x} v \cdot \bn \ \ds \quad \text{and} \quad g(x) = -v(x) \cdot \bt, \qquad x \in \partial \hat{K},
 \end{equation*}
    where the path integral is again taken in the counterclockwise direction. Thanks to the identity $\partial_{\bt} \bx^{\perp} = -\bn$ on $\partial \hat{K}$ and the divergence theorem, there holds
 \begin{equation*}
        (v , \bn )_{\partial \hat{K}} = -( v, \partial_{\bt} \bx^{\perp} )_{\partial \hat{K}} = ( \partial_{\bt} v, \bx^{\perp})_{\partial \hat{K}} = (\tau \bx^{\perp}, \bn )_{\partial \hat{K}} = (\dive \tau, \bx^{\perp})_{\hat{K}} = 0,
 \end{equation*}
    where we used that $\tau$ is symmetric and $\bx^{\perp} \in \rigidbodies$. Consequently,~\cref{eq:proof:one-integral-trace-reg} then gives
 \begin{equation}
        f \text{ is continuous}, \quad f|_{\hat{\gamma}_i} \in H^{s+5/2}(\hat{\gamma}_i), \quad \text{and} \quad g|_{\hat{\gamma}_i} \in H^{s+3/2}(\hat{\gamma}_i), \quad i \in \{1,2,3\},
 \end{equation}
    with the bound
 \begin{equation}\label{eq:proof:traces-edge-bound}
        \sum_{i=1}^{3} \{ \|f\|_{s+5/2, \hat{\gamma}_i}^2 + \|g\|_{s+3/2, \hat{\gamma}_i}^2 \} \leq C \sum_{i=1}^{3} \|v\|_{s+3/2, \hat{\gamma}_i}^2 \leq C \| \tau \|_{s+1, \hat{K}}^2.
 \end{equation}
    Moreover, $\mathfrak{D}(f, g) = \partial_{\bt} f \bt + g \bn = -v^{\perp} = (v_2, -v_1)^{\top}$, so $\mathfrak{D}(f, g)$ is continuous by~\cref{eq:proof:one-integral-trace-reg}. Thanks to the identity $\partial_{\bt} \mathfrak{D}(f, g) = \tau \bt$, there holds
 \begin{multline*}
        \bt_{i+1} \cdot \partial_{\bt}\mathfrak{D}(f, g)(\hat{a}_i + \delta \bt_{i+2}) - \bt_{i+2} \cdot \partial_{\bt} \mathfrak{D}(f, g)(\hat{a}_i - \delta \bt_{i+1}) \\ =  \bt_{i+1} \cdot \tau(\hat{a}_i + \delta \bt_{i+2}) \bt_{i+2} - \bt_{i+2} \cdot \tau(\hat{a}_i - \delta \bt_{i+1}) \bt_{i+1} 
 \end{multline*}
    for $0 \leq \delta \leq 1$ and $i \in \{1,2,3\}$. Consequently, if $s = 0$,
 \begin{equation}\label{eq:proof:traces-vertex-bound}
        \sum_{i=1}^{3} \mathcal{I}^2_i(\mathfrak{D}(f, g)) \leq \sum_{i=1}^{3}\int_{0}^{1} \frac{|\tau(\hat{a}_i + \delta \bt_{i+2}) - \tau(\hat{a}_i - \delta \bt_{i+1}) |^2}{\delta} \D \delta \leq \|\tau\|_{1/2, \partial \hat{K}} < \infty,
 \end{equation}
    while if $s > 0$, $\tau$ is continuous and so~\cref{eq:proof:traces-mixed-derivative} holds. Thus, $(f, g) \in X^{s+3}(\partial \hat{K})$.
    
    We then define the operator $\mathcal{L}_{\bn}$ by the rule $\tau \mapsto \mathcal{L}_{\bn} \tau := \tilde{\mathcal{L}}(f, g).$
    The mapping $\tau \mapsto (f, g)$ is linear by construction and so $\mathcal{L}_{\bn}$ is linear. Moreover,~\cref{eq:proof:traces-edge-bound,eq:proof:traces-vertex-bound} give
 \begin{equation*}
        \| \mathcal{L}_{\bn} \tau \|_{s+3, \hat{K}} \leq C \|(f, g)\|_{X^{s+3}(\partial \hat{K}; \mathbb{S})} \leq C \|\tau\|_{s+1, \hat{K}},
 \end{equation*}
    while direct computation shows that
 \begin{equation*}
        (\airy \mathcal{L}_{\bn} \tau) \bn = \partial_{\bt} (\partial_{\bt} f \bt + g \bn)^{\perp} = \partial_{\bt} v = \tau \bn \qquad \text{on } \partial \hat{K}, 
 \end{equation*}
    which completes the proof of~\cref{eq:airy-correction-properties}.
    
    Now suppose that $\tau \in \mathcal{P}_{p}(\hat{K}; \mathbb{S})$ with $\dive \tau \in \ltwomodrm(\hat{K}; \mathbb{V})$. By construction, there holds $f|_{\hat{\gamma}_i} \in \mathcal{P}_{p+2}(\hat{\gamma}_i)$ and $g|_{\hat{\gamma}_i} \in \mathcal{P}_{p+1}(\hat{\gamma}_i)$ for $1 \leq i \le 3$.
    Since $\tau \in X^s(\hat{K}; \mathbb{S})$ for all $s \geq 0$, $f$ and $\mathfrak{D}(f, g)$ are continuous and~\cref{eq:proof:traces-mixed-derivative} holds. By~\cite[Theorem 3.2]{Parker2023}, we have $\tilde{\mathcal{L}}(f, g) \in \mathcal{P}_{p+2}(\hat{K})$ and~\cref{eq:airy-correction-properties} follows.
\end{proof}

 \begin{lemma}\label{lem:reference-div-inversion}
    \emph{($p$-stable inversion of the divergence on the reference cell.)}
    There exists a linear operator $\tilde{\mathcal{L}}_{\dive} : \ltwomodrm(\hat{K}; \mathbb{V}) \to H^1(\hat{K}; \mathbb{S}) \cap \hdivz{\hat{K}}{\mathbb{S}}$ that maps $H^s(\hat{K}; \mathbb{V}) \cap \ltwomodrm(\hat{K}; \mathbb{V})$ boundedly into $H^{s+1}(\hat{K}; \mathbb{S}) \cap \hdivz{\hat{K}}{\mathbb{S}}$, for all real $s \geq 0$, and satisfies
 \begin{align}
        \label{eq:reference-div-inversion}
        \dive \tilde{\mathcal{L}}_{\dive} = I \quad \text{and} \quad \| \tilde{\mathcal{L}}_{\dive} u \|_{s+1, \hat{K}} \leq C \|u\|_{s, \hat{K}} \quad \forall~u \in H^s(\hat{K}; \mathbb{V}) \cap \ltwomodrm(\hat{K}; \mathbb{V}),
 \end{align}
    where $C$ depends only on $s$. Moreover, if $u \in \mathcal{P}_{p}(\hat{K}; \mathbb{V}) \cap \ltwomodrm(\hat{K}; \mathbb{V})$, $p \geq 0$, then $\tilde{\mathcal{L}}_{\dive} u \in \mathcal{P}_{p+1}(\hat{K}; \mathbb{S}) \cap \hdivz{\hat{K}}{\mathbb{S}}$.
 \end{lemma}
\begin{proof}
    Let $\mathcal{L}_{\div}$ be the operator in~\cref{lem:invert-div-ref-element-nobc} restricted to the space $\ltwomodrm(\hat{K}; \mathbb{V})$ and $\mathcal{L}_{\bn}$ the operator in~\cref{lem:airy-correction-operator}. Since $\dive \mathcal{L}_{\div} = I$, the composition $\mathcal{L}_{\bn} \circ  \mathcal{L}_{\div}$ is well-defined. Consequently, we define $\tilde{\mathcal{L}}_{\dive} : \ltwomodrm(\hat{K}; \mathbb{V}) \to L^2(\hat{K}; \mathbb{S})$ by the 
    divergence-free boundary correction $\tilde{\mathcal{L}}_{\dive} := (I - \airy \mathcal{L}_{\bn} )\mathcal{L}_{\div}.$

    Let $s \geq 0$ and $u \in H^s(\hat{K}; \mathbb{V}) \cap \ltwomodrm(\hat{K}; \mathbb{V})$. Inequalities \cref{eq:airy-correction-properties} and \cref{eq:invert-div-ref-element-nobc} then give
 \begin{equation*}
        \| \tilde{\mathcal{L}}_{\dive} u\|_{s+1, \hat{K}} \leq \|  \mathcal{L}_{\div} u \|_{s+1, \hat{K}} + \|\mathcal{L}_{\bn} \mathcal{L}_{\div} u\|_{s+3, \hat{K}} \leq C \|u\|_{s, \hat{K}}.
 \end{equation*}
    Consequently, $\tilde{\mathcal{L}}_{\dive} u \in H^{s+1}(\hat{K}; \mathbb{V})$. Moreover, $\dive \tilde{\mathcal{L}}_{\dive} u = \dive \mathcal{L}_{\div} u = u$ and
 \begin{equation*}
        (\tilde{\mathcal{L}}_{\dive} u) \bn = (\mathcal{L}_{\div} u) \bn  - (\airy \mathcal{L}_{\bn} \mathcal{L}_{\div} u) \bn  = 0 \qquad \text{on } \partial \hat{K}
 \end{equation*}
    by~\cref{eq:airy-correction-properties}. Thus, $\tilde{\mathcal{L}}_{\dive}$ maps $H^s(\hat{K}; \mathbb{V}) \cap \ltwomodrm(\hat{K}; \mathbb{V})$ boundedly into $H^{s+1}(\hat{K}; \mathbb{S}) \cap \hdivz{\hat{K}}{\mathbb{S}}$ and~\cref{eq:reference-div-inversion} holds. Moreiver, if $u \in \mathcal{P}_{p}(\hat{K}; \mathbb{V}) \cap \ltwomodrm(\hat{K}; \mathbb{V})$, $p \geq 0$, then $\mathcal{L}_{\div} u \in \mathcal{P}_{p+1}(\hat{K}; \mathbb{S})$ by~\cref{lem:invert-div-ref-element-nobc} and $\mathcal{L}_{\bn} \mathcal{L}_{\div} u \in \mathcal{P}_{p+3}(\hat{K})$ by~\cref{lem:airy-correction-operator}. Consequently, $\tilde{\mathcal{L}}_{\dive} u \in \mathcal{P}_{p+1}(\hat{K}; \mathbb{S}) \cap \hdivz{\hat{K}}{\mathbb{S}}$.
\end{proof}

 
\subsection{Proof of~\cref{thm:poincare-operators-element}}\label{sec:proof-poincare-element}

By~\cref{lem:exact-sequence-continuous-disp}, we have that~\cref{eq:exact-sequence-continuous-gen} is exact and so there exists a linear operator
 \begin{equation*}
    \airy^{-1} : \{ \tau \in \hdivz{\hat{K}}{\mathbb{S}} : \dive \tau \equiv 0 \} \to H_0^2(\hat{K})
 \end{equation*}
satisfying $\airy^{-1} \airy = I$ and $\airy \airy^{-1} = I$. Let $\sigma \in H^s(\hat{K}; \mathbb{S}) \cap \{ \tau \in \hdivz{\hat{K}}{\mathbb{S}} : \dive \tau \equiv 0 \}$; then there holds
 \begin{equation}\label{eq:proof:airy-inverse-bounded}
    \| \airy^{-1} \sigma \|_{s+2, \hat{K}} \leq C \|\airy \airy^{-1} \tau\|_{s, \hat{K}} = C \|\tau\|_{s, \hat{K}},
 \end{equation}
where we used Poincar\'{e}'s inequality in the first step. Consequently, $\airy^{-1}$ is a bounded operator from $ H^s(\hat{K}; \mathbb{S}) \cap \{ \tau \in \hdivz{\hat{K}}{\mathbb{S}} : \dive \tau \equiv 0 \}$ to $ H^2_0(\hat{K})$. If, additionally, $\sigma \in \mathcal{P}_{p}(\hat{K}; \mathbb{S})$, then
 \begin{equation*}
    \airy^{-1} \sigma \in \{ \tau \in H^2_0(\hat{K}) : \airy \tau \in \mathcal{P}_{p}(\hat{K}; \mathbb{S}) \} = \mathcal{P}_{p+2}(\hat{K}) \cap H^2_0(\hat{K}).
 \end{equation*}
Define $P_1 := \airy^{-1}(I - P_2 \dive)$ and $P_2 = \tilde{\mathcal{L}}_{\div}$.
By~\cref{lem:reference-div-inversion}, $\dive P_2 = I$, and so $P_1$ is well-defined since $\dive(I - P_2 \dive) = 0$. The properties~\cref{eq:poincare-operators-element-homotopy,eq:poincare-operators-element-poly-pres} now follow by construction, while~\cref{eq:poincare-operators-element-cont} follows from~\cref{eq:proof:airy-inverse-bounded,eq:reference-div-inversion}. 
\hfill \proofbox

 
\section{Inverting divergence on a mesh}\label{sec:invert-div-mesh}

We now construct an inverse of the divergence operator on a mesh. For the results in~\cref{sec:commuting-projections,sec:stable-decomp}, we require an extension of~\cref{thm:invert-div-fem}. In particular, we fix the following basis for $\rigidbodies$
 \begin{equation}\label{eq:rm-basis}
    r_1 := (1~~0)^T, \quad r_2 := (0~~1)^T, \quad \text{and} \quad r_3 := x^{\perp}.
 \end{equation}
Then the main result of this section is the following:
 \begin{theorem}[Flux and divergence recovery.]\label{thm:invert-div-fem-gen}
    Let $p \geq 3$. For every $u \in \ltwospace_{\Gamma}^{p-1}$ and $\omega \in \mathbb{R}^{J_N \times 3}$ satisfying
 \begin{equation}\label{eq:omega-sum-conditions}
        \sum_{j=1}^{J_N} \omega_{j, l} = \int_{\Omega} u \cdot r_l \ \dx, \qquad 1 \leq l \leq 3,
 \end{equation}
    there exists $\sigma \in \hdivspace_{\Gamma}^{p}$ such that
 \begin{subequations}
        \label{eq:global-div-inverse-gen}
        \begin{alignat}{2}
            \dive \sigma &= u, \qquad & & \\
            \int_{\Gamma_{N, j}} (\sigma \bn) \cdot r_l~\ds &= \omega_{j, l}, \qquad & & 1 \leq j \leq J_N, \ 1 \leq l \leq 3, \\
            \|\sigma\|_{\dive, \Omega} &\leq C( \|u\|_{0, \Omega} + |\omega|), \qquad & &
        \end{alignat}
 
    \end{subequations}
    where $\{ \Gamma_{N, j} : 1 \leq j \leq J_N \}$ are the $J_N$ connected components of $\Gamma_N$, $|\omega| = \sum_{j ,l} |\omega_{jl}|$,  and $C$ is independent of $h$ and $p$.
 \end{theorem}
The proof of~\cref{thm:invert-div-fem-gen} is in~\cref{sec:proof-invert-div-fem-gen}.

 
\subsection{Auxiliary coarse interpolant}

The following result will be useful in constructing an inverse of the divergence operator on a mesh:
 \begin{lemma}\label{lem:coarse-div-interp}
    Let $p \geq 3$. For every $u \in \ltwospace_{\Gamma}^{p-1}$ and $\omega \in \mathbb{R}^{J_N \times 3}$ satisfying~\cref{eq:omega-sum-conditions}, there exists $\sigma \in \hdivspace_{\Gamma}^{3}$ such that
 \begin{subequations}
        \label{eq:coarse-div-interp}
        \begin{alignat}{2}
        \int_{K} (\dive \sigma) \cdot r~\dx &= \int_{K} u \cdot r~\dx \qquad & &\forall~r \in \rigidbodies, \ \forall~K \in \mathcal{T}, \\
        \int_{\Gamma_{N, j}} (\sigma \bn) \cdot r_l \ \ds &= \omega_{j, l}, \qquad & &1 \leq j \leq J_N, \ 1 \leq l \leq 3, \\ 
         \|\sigma\|_{\dive, \Omega} &\leq C \left( \|u\|_{0, \Omega} + |\omega| \right) , \qquad & &
        \end{alignat}
 \end{subequations}
    where $C$ is independent of 
    $h$ and $p$.
 \end{lemma}
\begin{proof}
    Let $u \in \ltwospace_{\Gamma}^{p-1}$ and $\omega \in \mathbb{R}^{J_N \times 3}$ satisfying~\cref{eq:omega-sum-conditions} be given. Thanks to~\cref{cor:h1-inversion-div-bc-avg}, there exists $\tau \in H^1(\Omega; \mathbb{S}) \cap \hdivgamma{\Omega}{\mathbb{S}}$ satisfying $\dive \tau = u$, $\int_{\Gamma_{N, j}} (\tau \bn) \cdot r_l = \omega_{j, l}$ for $1 \leq j \leq J_N$, $1 \leq l \leq 3$, and $\|\tau\|_{1, \Omega} \leq C ( \|u\|_{0, \Omega} + |\omega|)$. Let $\tau_1 \in \{ \rho \in H^1(\Omega; \mathbb{S}) \cap \hdivgamma{\Omega}{\mathbb{S}}:\rho|_K \in \mathcal{P}_1(K; \mathbb{S}) \ \forall~K \in \mathcal{T} \}$ be the componentwise linear Scott--Zhang interpolant~\cite{Scott1990} of $\tau$, which satisfies
 \begin{equation}\label{eq:proof:scott zhang stability}
        \| \tau_1 \|_{1, K} + h_K^{-1/2} \| \tau - \tau_1 \|_{0, \partial K}  \leq C \|\tau\|_{1, \mathcal{T}_K} \qquad \forall~K \in \mathcal{T},
 \end{equation}
    where $\mathcal{T}_K$ is the patch of elements sharing an edge or vertex with $K$. 
    
    We now define $\rho \in \Sigma^{3}$ by assigning the following degrees of freedom:
 \begin{equation}\label{eq:proof:rho interp edge momoments}
        \int_{\gamma} (\rho \bn) \cdot r \ \ds = \int_{\gamma} (\tau - \tau_1) \bn \cdot r \ \ds \qquad \forall~r \in \mathcal{P}_{2}(\gamma; \mathbb{V}),
 \end{equation}
    while the remaining degrees of freedom in~\cref{lem:hdiv-global-dofs} vanish. In particular, $\rho \in \hdivspace_{\Gamma}^{3}$.
    A standard norm equivalence argument on the finite dimensional space $\mathcal{P}_{3}(K; \mathbb{S})$ gives
 \begin{align}
        h_K^{-1} \| \rho \|_{0, K} + \|\dive \rho\|_{0, K} \leq C h_K^{-1/2} \sup_{\substack{ \gamma \in \mathcal{E}_K \\ r \in \mathcal{P}_{2}(\gamma; \mathbb{V}) \\ \|r\|_{0, \gamma} = 1  } } \left| \int_{\gamma} (\rho \bn) \cdot r \ \ds \right| 
        &\leq C h_K^{-1/2} \| \tau - \tau_1\|_{0, \partial K}  \notag \\
        \label{eq:proof:rho continuity}
        &\leq C \|\tau\|_{1, \mathcal{T}_K}.
 \end{align}

    Let $\sigma := \tau_1 + \rho$. Then $\sigma \in \hdivspace_{\Gamma}^{3}$ 
    by construction, and using the divergence theorem and~\cref{eq:proof:rho interp edge momoments}, we obtain
 \begin{equation}
        \int_{K} \dive \sigma \cdot r~\dx = \int_{\partial K} \sigma \bn \cdot r~\ds =  \sum_{\gamma \in \mathcal{E}_K} \int_{\gamma} \tau \bn \cdot r~\ds = \int_{K} \dive \tau \cdot r~\dx = \int_{K} u \cdot r~\dx,
 \end{equation}
    for all $r \in \rigidbodies$, and $(\sigma \bn, r_l)_{\Gamma_{N, j}} = (\tau \bn, r_l)_{\Gamma_{N, j}} = \omega_{j, l}$
    for $1 \leq j \leq J_N$, $1 \leq l \leq 3$. Moreover, ~\cref{eq:proof:scott zhang stability,eq:proof:rho continuity} give
 \begin{equation*}
        \sum_{K \in \mathcal{T}} \|\sigma\|_{\dive, K}^2 \leq 2 \sum_{K \in \mathcal{T}} (\|\tau_1\|_{\dive, K}^2 + \|\rho\|_{\dive, K}^2) \leq C \sum_{K \in \mathcal{T}} \|\tau\|_{1,\mathcal{T}_K}^2 \leq C \|\tau\|_{1, \Omega}^2.
 \end{equation*}
  Properties \cref{eq:coarse-div-interp} now follow.
\end{proof}

 
\subsection{Proof of~\cref{thm:invert-div-fem-gen}}\label{sec:proof-invert-div-fem-gen}

Let $u \in \ltwospace_{\Gamma}^{p - 1}$ and $\omega \in \mathbb{R}^{J_N \times 3}$ satisfying~\cref{eq:omega-sum-conditions} be given. Let $\tau \in \hdivspace_{\Gamma}^{3}$ be given by~\cref{lem:coarse-div-interp}. On each $K \in \mathcal{T}$, the function $u_K := (u - \dive \tau)|_{K} \in \mathcal{P}_{p-1}(K; \mathbb{V}) \cap \ltwomodrm(K; \mathbb{V})$, and so there exists $\sigma_K \in \mathcal{P}_{p}(K; \mathbb{S}) \cap \hdivz{K}{\mathbb{S}}$ satisfying $\dive \sigma_K = u_K$ and $h_K^{-1} \|\sigma_K\|_{0,K} + |\sigma_K|_{1,K} \leq C \|u_K\|_{K}$ thanks to~\cref{cor:single-div-inversion}. 

Define $\sigma$ by the rule $\sigma|_{K} = \tau + \sigma_K$. Then $\sigma \in \Sigma_{\Gamma}^{p}$. Moreover, $\dive \sigma|_{K} = \dive \tau|_{K} + \dive \sigma_K = u|_{K}$ on $K \in \mathcal{T}$, while~\cref{eq:single-div-inversion} and shape regularity give
 \begin{align*}
    \sum_{K \in \mathcal{T}} \|\sigma\|_{0, K}^2 \leq C \sum_{K \in \mathcal{T}} (\|\tau\|_{0, K}^2 + h_K^2 \| u - \dive \tau\|_{0, K}^2) &\leq C \sum_{K \in \mathcal{T}} (\|\tau\|_{\dive, K}^2 + \|u\|_{0, K}^2) \\
    &\leq C (\|\tau\|_{\dive, \Omega}^2 + \|u\|_{0, \Omega}^2).
 \end{align*}
\Cref{eq:global-div-inverse-gen} now follows from~\cref{eq:coarse-div-interp}. 
\hfill\proofbox

 
\subsection{Proof of~\cref{thm:invert-div-fem}}

Let $u \in \ltwospace_{\Gamma}^{p-1}$ be given. If $|\Gamma_N| = 0$, then~\cref{thm:invert-div-fem-gen} is the same statement as~\cref{thm:invert-div-fem}. If $\Gamma_N$ has at least one connected component, then we set $\omega_{1, l} = \int_{\Omega} u \cdot r_l \ \dx$ and $\omega_{j, l} = 0$ for $2 \leq j \leq J_N$, $1 \leq l \leq 3$ and apply~\cref{thm:invert-div-fem-gen}. The inequality in~\cref{eq:global-div-interp} now follows on noting that $|\omega| \leq \|u\|_{0, \Omega}$. Consequently, for every $u \in \ltwospace_{\Gamma}^{p-1}$, there holds
 \begin{equation*}
    \sup_{ 0 \neq \sigma \in  \hdivspace_{\Gamma}^{p}} \frac{ (\dive \sigma, u) }{ \|\sigma\|_{\dive, \Omega} \|u\|_{0, \Omega}  } \geq C^{-1} > 0.
 \end{equation*}
Inequality~\cref{eq:inf-sup-def} now follows. 
\hfill\proofbox

 
\section{Bounded cochain projections and Hodge decompositions}\label{sec:cochain}

In this section, we prove~\cref{lem:sigma-projection-dofs},~\cref{thm:commuting-diagram}, and~\cref{thm:stable-decomposition-low-airy-div}. The first step is to construct a tensor in $\hdivspace_{\Gamma}^{p}$ with the given degrees of freedom in~\cref{lem:sigma-projection-dofs}:
 \begin{lemma}\label{lem:projection-matching}
    Let $p \geq 3$. For every $q \in \htwospace_{\Gamma}^{p+2}$, $v \in \ltwospace_{\Gamma}^{p-1}$, and $\kappa \in \mathbb{R}^{|\mathfrak{I}^*| \times 3}$, there exists a unique $\sigma \in \hdivspace_{\Gamma}^{p}$ such that
 \begin{subequations}\label{eq:projection-matching}
        \begin{align}
            (\sigma, \airy s) &= (\airy q, \airy s) \qquad & &\forall~s \in \htwospace_{\Gamma}^{p+2}, \\
            (\dive \sigma, u) &= (v, u) \qquad & & \forall~u \in \ltwospace_{\Gamma}^{p-1}, \\
            \langle \sigma \bn, r_l \rangle_{\partial \Omega_m} &= \kappa_{m, l} \qquad & &\forall~m \in \mathfrak{I}^*, \ 1 \leq l \leq 3,
        \end{align}
 \end{subequations}
    where $\{r_l\}$ are defined in~\cref{eq:rm-basis}. Moreover, $\sigma$ satisfies
 \begin{equation}\label{eq:projection-matching-stability}
        \| \sigma\|_{\dive, \Omega} \leq C \left( \|\airy q\|_{0, \Omega} + \|v\|_{0, \Omega} + |\kappa| \right),
 \end{equation}
    where $C$ is independent of $h$ and $p$. 
 \end{lemma}
\begin{proof}
    \textbf{Step 1: Existence.}
    Let $q \in \htwospace_{\Gamma}^{p+2}$, $v \in \ltwospace_{\Gamma}^{p-1}$, and $r_m \in \rigidbodies$, $m \in \mathfrak{I}^*$ be given. Thanks to~\cref{thm:invert-div-fem}, there exists $\rho \in \hdivspace_{\Gamma}^{p}$ such that $\dive \rho = v$ and $\|\rho\|_{\dive, \Omega} \leq C \|v\|_{0,\Omega}$. Let $\tau \in (N_{\Gamma}^{p})^{\perp}$ be the $(\cdot,\cdot)_{\dive}$-orthogonal projection of $\rho$ onto $(N_{\Gamma}^{p})^{\perp}$~\cref{eq:div-free-def}. Then $\dive \tau = \dive \rho = u$ and $\|\tau\|_{\dive, \Omega} \leq \|\rho\|_{\dive, \Omega} \leq C \|v\|_{0, \Omega}$.
    
    If $|\Gamma| = |\Gamma_D|$, then set $\sigma_3 = 0$. Otherwise,
    let $\omega \in \mathbb{R}^{J_N \times 3}$ be chosen so that for $1 \leq l \leq 3$, there holds
 \begin{equation}
 	\label{eq:proof:choosing-omegail}
        \sum_{ \substack{i=1 \\ \Gamma_{N, i} \subset \partial \Omega_m } }^{J_N} \omega_{i, l} = \kappa_{m, l} - \langle \tau \bn, r_l \rangle_{\partial \Omega_m} \quad \forall~m \in \mathfrak{I}^*, \qquad \sum_{ i=1 }^{J_N} \omega_{i, l} = 0,
 \end{equation}
    and  $|\omega| \leq C (|\kappa| + \|\tau\|_{\dive, \Omega})$.
    \Cref{thm:invert-div-fem-gen} shows that there exists $\sigma_3 \in \hdivspace_{\Gamma}^{3}$ such that $\dive \sigma_3 \equiv 0$,
 \begin{equation*}
        \langle \sigma_3 \bn, r_l \rangle_{\partial \Omega_m} = \sum_{ \substack{i=1 \\ \Gamma_{N, i} \subset \partial \Omega_m } }^{J_N} \omega_{i, l} = \kappa_{m, l} - \langle \tau \bn, r_l \rangle_{\partial \Omega_m}
         \qquad \forall~m \in \mathfrak{I}^*, \ 1 \leq l \leq 3,
 \end{equation*}
    and $\|\sigma_3\|_{\dive, \Omega} \leq C |\omega| \leq C (|\kappa| + \|\tau\|_{\dive, \Omega})$.
    
    Since $\airy \htwospace_{\Gamma}^{p+2}$ is a closed subspace of $N_{\Gamma}^p$~\cref{eq:div-free-def}, there exists a unique $\phi \in N_{\Gamma}^p$ satisfying
 \begin{equation*}
        \| \phi \|_{\dive, \Omega} \leq \|\sigma_3\|_{\dive, \Omega}, \quad  
        \phi - \sigma_3 \in \airy \htwospace_{\Gamma}^{p+2} \quad \text{and} \quad (\phi, \airy s) = 0 \qquad \forall~s \in \htwospace_{\Gamma}^{p+2}.
 \end{equation*}
    Moreover, $\langle (\airy s) \bn, r \rangle_{\partial \Omega_m} = 0$ for all $s \in \htwospace_{\Gamma}^{p+2}$, $r \in \rigidbodies$, $0 \leq m \leq M$, and so $\langle \phi \bn, r_l \rangle_{\partial \Omega_m} = \kappa_{m, l} - \langle \tau \bn, r_l \rangle_{\partial \Omega_m}$ for $m \in \mathfrak{I}^*$, $1 \leq l \leq 3$. The tensor $\sigma = \tau + \phi + \airy q$ then satisfies~\cref{eq:projection-matching,eq:projection-matching-stability} by construction.
    
    \textbf{Step 2: Uniqueness.}
    Now suppose that $\sigma \in \hdivspace_{\Gamma}^{p}$ satisfies~\cref{eq:projection-matching} with $q \equiv 0$, $v \equiv 0$, and $\kappa \equiv 0$. Since $p \geq 3$, there holds $\ltwospace_{\Gamma}^{p-1} = \dive \hdivspace_{\Gamma}^{p}$ by~\cref{thm:invert-div-fem}, and so the condition $(\dive \sigma, v) = 0$ for all $v \in \ltwospace_{\Gamma}^{p-1}$ means that $\dive \sigma \equiv 0$. Moreover, the condition $\langle \sigma \bn, r \rangle_{\partial \Omega_m} = 0$ for all $r \in \rigidbodies$, $m \in \mathfrak{I}^*$ gives
 \begin{equation*}
        \langle \sigma \bn, r \rangle_{\partial \Omega_n} = \langle \sigma \bn, r \rangle_{\partial \Omega_n} + \langle \sigma \bn, r \rangle_{\Gamma_D} + \sum_{m \in \mathfrak{I}^*} \langle \sigma \bn, r \rangle_{\partial \Omega_m}  = \langle \sigma \bn, r \rangle_{\partial \Omega} = (\dive \sigma, r)_{\Omega} = 0, 
 \end{equation*}
    where $n$ is the sole element in $\mathfrak{I} \setminus \mathfrak{I}^*$ (provided that $\mathfrak{I}$ is nonempty). Consequently, $\langle \sigma \bn, r \rangle_{\partial \Omega_m} = 0$ for all $0 \leq m \leq M$.
    
    By~\cite[Theorem 2]{Geymonat2000}, there exists $\tilde{q} \in H^2(\Omega)$ such that $\airy \tilde{q} = \sigma$. Since $\airy \tilde{q}$ is a piecewise polynomial of degree $p$, there holds $\tilde{q} \in \htwospace^{p+2}$. Since $\langle (\airy \tilde{q})\bn, v \rangle_{\partial \Omega} = 0$ for all $v \in H^1(\Omega)$ with $v|_{\Gamma_N} = 0$ and $(\airy \tilde{q}) \bn = \partial_{\bt} ( (\partial_{\bn} \tilde{q}) \bt - (\partial_{\bt} \tilde{q}) \bn ) = \partial_{\bt} (\grad \tilde{q})^{\perp}$, we have that $\grad \tilde{q}|_{\Gamma_{D, j}} \in \mathbb{R}$ for $0 \leq j \leq J_D$. In particular, there exists an affine function $\ell \in \mathcal{P}_1(\Omega)$ such that $q := \tilde{q} - \ell$ satisfies $q|_{\Gamma_{D, 0}} = \partial_{\bn} q|_{\Gamma_{D, 0}} = 0$ so that $q \in \htwospace_{\Gamma}^{p+2}$. Thus, $\sigma = \airy q$ and the condition $(\sigma, \airy s) = (\airy q, \airy s)$ for all $s \in \htwospace_{\Gamma}^{p+2}$, gives $q \equiv 0$, and so $\sigma \equiv 0$. The uniqueness of $\sigma$ satisfying~\cref{eq:projection-matching} now follows.
\end{proof}

 
\subsection{Proof of~\cref{lem:sigma-projection-dofs}}\label{sec:proof-projection-dofs}

Suppose that $\sigma \in \hdivspace_{\Gamma}^{p}$ and that all of the degrees of freedom in~\cref{lem:sigma-projection-dofs} vanish. Inequality~\cref{eq:projection-matching-stability} then gives that $\sigma \equiv 0$ and so $\dim \hdivspace_{\Gamma}^p \leq \dim \airy \htwospace_{\Gamma}^{p + 2} + \dim \ltwospace_{\Gamma}^{p - 1} + 3|\mathfrak{I}^*|$.

On the other hand,~\cref{lem:projection-matching} shows that the degrees of freedom in the statement of the lemma are linearly independent and that $\dim \hdivspace_{\Gamma}^p \geq \dim \airy \htwospace_{\Gamma}^{p + 2} + \dim \ltwospace_{\Gamma}^{p - 1} + 3|\mathfrak{I}^*|$. The result now follows since 
$\dim \airy \htwospace_{\Gamma}^{p + 2} = \dim \htwospace_{\Gamma}^{p + 2} - \dim \mathcal{P}_{1, \Gamma}(\Omega)$. 
\hfill\proofbox

 
\subsection{Proof of~\cref{thm:commuting-diagram}}\label{sec:proof-commuting-diagram}

Let $p \geq 3$ be given.

\textbf{Step 1: Projection.} The property that $\Pi_{\htwospace}^{p+2} \circ \Pi_{\htwospace}^{p+2} = \Pi_{\htwospace}^{p+2}$ and $\Pi_{\ltwospace}^{p-1} \circ \Pi_{\ltwospace}^{p-1} = \Pi_{\ltwospace}^{p-1}$ follows by construction. For $\sigma \in \hdivspace_{\Gamma}^{p}$, the functions $\Pi_{\hdivspace}^{p} \sigma$ and $\sigma$ have the same degrees of freedom listed in~\cref{lem:sigma-projection-dofs}, and so $\Pi_{\hdivspace}^{p} \sigma = \sigma$ and $\Pi_{\hdivspace}^{p} \circ \Pi_{\hdivspace}^{p} = \Pi_{\hdivspace}^{p}$.

\textbf{Step 2: Stability.} Since $\Pi_{\ltwospace}^{p-1}$ is an orthogonal projection, inequality~\cref{eq:pi-v-bounded} holds, while inequality~\cref{eq:pi-q-bounded} is an immediate consequence of  Poincar\'{e}'s inequality.  

Now suppose that $\sigma \in \hdivgamma{\Omega}{\mathbb{S}}$ and let $q \in \htwospace_{\Gamma}^{p+2}$ be given by
 \begin{equation*}
    (\airy q, \airy r) = (\sigma, \airy r) \qquad \forall~r \in \htwospace_{\Gamma}^{p+2} \quad \text{and} \quad (q, \ell) = 0 \qquad \forall~\ell \in \mathcal{P}_{1, \Gamma}(\Omega).
 \end{equation*}
By construction, $\|\airy q\|_{0, \Omega} \leq C \|\sigma\|_{0, \Omega}$. Moreover, $\Pi_{\hdivspace}^{p} \sigma$ satisfies
 \begin{alignat*}{2}
    (\Pi_{\hdivspace}^p \sigma, \airy s) &= (\airy q, \airy s) \qquad & &\forall~q \in \htwospace_{\Gamma}^{p+2}, \\
    (\dive \Pi_{\hdivspace}^p \sigma , v) &= (\Pi_{\ltwospace}^{p-1} \dive \sigma, v) \qquad & &\forall~v \in \ltwospace_{\Gamma}^{p-1}, \\
    \langle \Pi_{\hdivspace}^{p} \sigma \bn, r \rangle_{\partial \Omega_m} &= \langle \sigma \bn, r \rangle_{\partial \Omega_m}, \qquad & &\forall~r \in \rigidbodies, \ \forall~m \in \mathfrak{I}^*,
 \end{alignat*}
and so inequality~\cref{eq:pi-sigma-bounded} follows from~\cref{eq:projection-matching-stability}. 

\textbf{Step 3: Commutativity.} Let $q \in H^2_{\Gamma}(\Omega)$ be given. By definition~\cref{eq:pi-sigma-def}, there holds
 \begin{equation*}
    (\Pi_{\hdivspace}^p \airy q, \airy s) = (\airy q, \airy s) = (\airy \Pi_{\htwospace}^{p+2} q, \airy s ) \qquad  \forall~s \in \htwospace_{\Gamma}^{p},
 \end{equation*}
while $\dive \Pi_{\hdivspace}^p \airy q \equiv 0 \equiv \dive \airy \Pi_{\htwospace}^{p+2} q$ and 
 \begin{equation*}
    \langle \Pi_{\hdivspace}^p \airy q, r \rangle_{\partial \Omega_m} = \langle \airy q, r \rangle_{\partial \Omega_m} = 0 = \langle \airy \Pi_{\htwospace}^{p+2} q, r \rangle_{\partial \Omega_m} \qquad \forall~r \in \rigidbodies, \ \forall~m \in \mathfrak{I}^*.
 \end{equation*}
As a result, $\Pi_{\hdivspace}^p \airy q \in \hdivspace_{\Gamma}^{p}$ and $\airy \Pi_{\htwospace}^{p+2} q \in \hdivspace_{\Gamma}^{p}$ have the same degrees of freedom listed in~\cref{lem:sigma-projection-dofs}, and thus~\cref{eq:commute-airy} holds. \Cref{eq:commute-div} follows immediately by construction~\cref{eq:pi-sigma-def}. 
\hfill\proofbox

 
\subsection{Proof of~\cref{thm:stable-decomposition-low-airy-div}}\label{sec:proof-stable-decomp}

\textbf{Step 1: $\tau$.} Let $\sigma \in \hdivspace_{\Gamma}^{p}$ be given. The same arguments in Step 1 of the proof of~\cref{lem:projection-matching} shows that there exists $\tau \in (N_{\Gamma}^p)^{\perp}$ satisfying $\dive \tau = \dive \sigma$ and $\|\tau\|_{\dive, \Omega} \leq C \|\dive \sigma\|_{0, \Omega}$.

\textbf{Step 2: $\phi_3$.} The function $\tilde{\sigma} := \sigma - \tau$ satisfies $\dive \tilde{\sigma} \equiv 0$ and $\|\tilde{\sigma}\|_{\dive, \Omega} \leq C \|\sigma\|_{\dive, \Omega}$. Applying~\cref{lem:projection-matching} with $p=3$, $q \equiv 0$, $v \equiv 0$, and $\kappa_{m, l} := \langle \tilde{\sigma} \bn, r_l \rangle_{\partial \Omega_m}$, $m \in \mathfrak{I}^*$, $1 \leq l \leq 3$, we obtain $\phi_3 \in \mathfrak{H}_{\Gamma}^3$~\cref{eq:harmonic-forms-discrete-def} satisfying $\dive \phi_3 \equiv 0$ and
 \begin{align}\label{eq:proof:sigma3-boundary-moments}
    \langle \phi_3 \bn, r \rangle_{\partial \Omega_m} = \langle \tilde{\sigma} \bn, r \rangle_{\partial \Omega_m}, \qquad \forall~r \in \rigidbodies, \ \forall~m \in \mathfrak{I}^*.
 \end{align}
Arguing as in Step 2 of the proof of~\cref{lem:projection-matching}, we may show that~\cref{eq:proof:sigma3-boundary-moments} holds for all $0 \leq m \leq M$.

\textbf{Step 3: $q$.} The function $\rho := \tilde{\sigma} - \phi_3$ then satisfies $\dive \rho \equiv 0$ and $\langle \phi_3 \bn, r \rangle_{\partial \Omega_m} = 0$ for all $r \in \rigidbodies$ and $1 \leq m \leq M$. The same arguments in Step 2 of the proof of~\cref{lem:projection-matching} show that there exists $q \in Q_{\Gamma}^{p+2}$ such that $\rho = \airy q$. If $|\Gamma_D| > 0$, then $\|q\|_{2, \Omega} \leq C \|\airy q\|_{0, \Omega}$ by Poincar\'{e}'s inequality. If $|\Gamma_D| = 0$, then $q$ is unique up to an affine function, so we choose $q$ so that $\|q\|_{2, \Omega} \leq C \|\airy q\|_{0, \Omega}$. In both cases,~\cref{eq:low-order-div-free-high-order-decomp} now follows on collecting results. 
\hfill\proofbox

\appendix

 
\section{Equivalence to the Hu--Zhang element}\label{sec:hu-zhang}

The original Hu--Zhang work~\cite{Hu2014} characterized the local space 
via primal basis functions,
as 
constant tensors in a basis of $\mathbb{S}$ 
multiplied by standard $C^0$ Lagrange elements, and augmented with edge bubbles.

 \begin{lemma}\label{lem:equal-hu-zhang}
    \emph{(Coincidence with the Hu--Zhang space.)}
    Let $p \geq 3$. Given $K \in \mathcal{T}$ labeled as in~\cref{fig:general-triangle}, let $\{ \lambda_i \}_{i=1}^{3}$ denote the barycentric coordinates on $K$ and define
 \begin{equation*}
        \hdivspace_{\partial K} := \spann\{ (\lambda_{i} \lambda_{i+1} q) \bt_{i+2} \otimes\bt_{i+2} : q \in \mathcal{P}_{p-2}(\gamma_i), \ 1 \leq i \leq 3   \},
 \end{equation*}
    where the indices are understood modulo 3. Then the space $\hdivspace^p$ admits the following alternative characterization:
 \begin{align}\label{eq:hu-zhang-def}
        \begin{aligned}
            \hdivspace^p &= \{ \sigma \in \hdiv{\Omega}{\mathbb{S}}:\sigma = \sigma_c + \sigma_b, \ \sigma_c \in H^1(\Omega; \mathbb{S}), \\
            &\qquad \sigma_c|_{K} \in \mathcal{P}_p(K; \mathbb{S}) \quad \text{and} \quad \sigma_b|_{K} \in \hdivspace_{\partial K} \ \forall~K \in \mathcal{T} \}.
        \end{aligned}
 \end{align}
    Consequently, $\Sigma^{p}$ is the same stress space as appears in~\cite{Hu2014}.
 \end{lemma}
\begin{proof}
    Let $\Upsilon^p$ denote the space on the right of~\cref{eq:hu-zhang-def}. As shown in~\cite{Hu2014}, $\Upsilon^p \subset \hdiv{\Omega}{\mathbb{S}}$ and so $\Upsilon^p \subseteq \Sigma^p$. The basis of $\Upsilon^p$ given on~\cite[p.~5]{Hu2014} consists of
 \begin{equation*}
        3 |\mathcal{V}| + 2(p-1)|\mathcal{E}| + \frac{3(p-1)(p-2)}{2} |\mathcal{T}| + (p-1)(2|\mathcal{E}_I| + |\mathcal{E}_B|)
 \end{equation*}
    functions, where $\mathcal{E}_I$ and $\mathcal{E}_B$ denote the number of interior and boundary edges. Using the identity $2|\mathcal{E}_I| + |\mathcal{E}_B| = 3|\mathcal{T}|$ shows that
 \begin{equation*}
        \dim \Upsilon^p = 3|\mathcal{V}| +  2(p-1)|\mathcal{E}| + \frac{3}{2}p(p-1) = \dim \Sigma^p
 \end{equation*}
    thanks to~\cref{lem:hdiv-global-dofs}, and thus $\Upsilon^p = \Sigma^p$.
\end{proof}

 
\section{Properties of the continuous elasticity complex}\label{sec:exactness-stability-continuous}

In this section, we prove~\cref{thm:h1-inversion-div-bc} and the continuous analog of~\cref{thm:invert-div-fem}. We first construct some inverse trace operators in~\cref{sec:cont-trace-results} and then construct an inverse for the divergence operator in~\cref{sec:cont-inv-div}.

 
\subsection{Trace results}\label{sec:cont-trace-results}

We begin with an analog of~\cref{lem:airy-correction-operator} for a simply connected polygonal domain:
 \begin{lemma}\label{lem:airy-correction-omega}
    Let $U \subset \mathbb{R}^2$ be a simply connected polygonal domain. For all $\tau \in H^1(U; \mathbb{S})$ such that $\dive \tau \in \ltwomodrm(U; \mathbb{V})$, there exists $q \in H^3(U)$ satisfying
 \begin{equation*}
        (\airy q)\bn = \tau \bn \qquad \text{on } \partial U \quad \text{and} \quad \|q\|_{3, U} \leq C \|\tau\|_{1, U},
 \end{equation*}
    where $C$ depends only on $U$.
 \end{lemma}
\begin{proof}
    The proof proceeds similarly as the proof of~\cref{lem:airy-correction-operator}. Let $\tau$ be as in the statement of the lemma and let $\{ A_j \}_{j=1}^N$ and $\{ \Gamma_j \}_{j=1}^N$ denote the vertices and edges of $U$ in a counterclockwise ordering as in~\cref{fig:domain-example}. Since $U$ is simply connected, we define $v \in L^2(\partial U; \mathbb{V})$ by the rule
 \begin{equation*}
        v(x) = \int_{{A}_1}^{x} \tau \bn \ \ds, \qquad x \in \partial U,
 \end{equation*}
    where the path integral is taken counterclockwise around $\partial U$. Analogous arguments as in the proof of~\cref{lem:airy-correction-operator} show that
 \begin{equation}\label{eq:proof:one-integral-trace-reg-omega}
        v \text{ is continuous} \quad \text{and} \quad v|_{\hat{\Gamma}_i} \in H^{3/2}(\Gamma_i; \mathbb{V}), \quad 1 \leq i \leq N.
 \end{equation}
    Now define the functions $f, g : \partial U \to \mathbb{R}$ by the rules
 \begin{equation*}
        f(x) = \int_{{A}_1}^{x} v \cdot \bn \ \ds \quad \text{and} \quad g(x) = -v(x) \cdot \bt, \qquad x \in \partial U,
 \end{equation*}
    where the path integral is again taken in the counterclockwise direction. Arguing as in the proof of~\cref{lem:airy-correction-operator} shows that
 \begin{equation*}
        f \text{ is continuous}, \quad f|_{{\Gamma}_i} \in H^{5/2}({\Gamma}_i), \quad \text{and} \quad g|_{{\Gamma}_i} \in H^{3/2}({\Gamma}_i), \quad 1 \leq i \leq N,
 \end{equation*}
    with the bound
 \begin{equation}\label{eq:proof:traces-edge-bound-omega}
        \sum_{i=1}^{N} \{ \|f\|_{5/2, {\Gamma}_i}^2 + \|g\|_{3/2, {\Gamma}_i}^2 \} \leq C \sum_{i=1}^{N} \|v\|_{3/2, {\Gamma}_i}^2 \leq C \| \tau \|_{1, \Omega}^2.
 \end{equation}
    Moreover, $\mathfrak{D}(f, g) = \partial_{\bt} f \bt + g \bn = -v^{\perp}$, so $\mathfrak{D}(f, g)$ is continuous by~\cref{eq:proof:one-integral-trace-reg-omega}. Let $\bt_{i}$ and $\bn_i$ denote the unit tangent and outward unit normal vectors on $\Gamma_i$ and $\epsilon := \min_{1 \leq i \leq N} |\Gamma_i|$. Thanks to the identity $\partial_{\bt} \mathfrak{D}(f, g) = \tau \bt$, there holds
 \begin{multline*}
        \bt_{i+1} \cdot \partial_{\bt}\mathfrak{D}(f, g)({A}_i + \delta \bt_{i+2}) - \bt_{i+2} \cdot \partial_{\bt} \mathfrak{D}(f, g)({A}_i - \delta \bt_{i+1}) \\ =  \bt_{i+1} \cdot \tau({A}_i + \delta \bt_{i+2}) \bt_{i+2} - \bt_{i+2} \cdot \tau({A}_i - \delta \bt_{i+1}) \bt_{i+1} 
    \end{multline*}
    for $0 \leq \delta \leq \epsilon$ and $1 \leq i \leq N$. Consequently,
 \begin{multline}\label{eq:proof:traces-vertex-bound-omega}
        \sum_{i=1}^{N} \int_{0}^{\epsilon} \frac{|\bt_{i+1} \cdot \partial_{\bt}\mathfrak{D}(f, g)({A}_i + \delta \bt_{i+2}) - \bt_{i+2} \cdot \partial_{\bt} \mathfrak{D}(f, g)({A}_i - \delta \bt_{i+1})|^2}{\delta} \ \D \delta  \\
         \leq \sum_{i=1}^{N}\int_{0}^{\epsilon} \frac{|\tau(\hat{a}_i + \delta \bt_{i+2}) - \tau(\hat{a}_i - \delta \bt_{i+1}) |^2}{\delta} \D \delta \leq \|\tau\|_{1/2, \partial \Omega}^2 < \infty.
 \end{multline}
    Thanks to~\cite[Theorem 6.1]{Arnold1988}, there exists $q \in H^3(U)$ satisfying $q|_{\partial U} = f$, $\partial_{\bn} q|_{\partial U} = g$ and the bound $\|q\|_{3, U} \leq C \|\tau\|_{1, U}$ follows from~\cref{eq:proof:traces-edge-bound-omega,eq:proof:traces-vertex-bound-omega}. Additionally, $(\airy q)\bn = \tau \bn$ on $\partial U$, which completes the proof.
\end{proof}

The second result concerns the existence of symmetric $H^1$ tensor fields with specified moments on $\Gamma_N$:
 \begin{lemma}\label{lem:h1-rigid-interp}
    For every $\omega \in \mathbb{R}^{J_N \times 3}$, there exists $\Psi \in H^1(\Omega; \mathbb{S}) \cap \hdivgamma{\Omega}{\mathbb{S}}$ such that
 \begin{equation}\label{eq:h1-rigid-interp}
        \int_{ \Gamma_{N, j} } (\Psi \bn) \cdot r_l \ \ds = \omega_{j, l}, \quad 1 \leq j \leq J_N, \ 1 \leq l \leq 3, \quad \text{and} \quad \| \Psi \|_{1} \leq C |\omega|,
 \end{equation}
    where $C$ is independent of $\omega$.
 \end{lemma}
\begin{proof}
    Let $1 \leq j \leq J_N$ be given and let $k(j) \in \{1, \ldots, N\}$ be any index such that the edge $\Gamma_{k(j)} \subseteq \Gamma_{N, j}$. For $1 \leq l \leq 3$, let $\phi_{j,l} \in C^{\infty}_c(\Gamma_{k(j)}; \mathbb{V})$ be any functions satisfying $(\phi_{j,l}, r_n)_{\Gamma_{k(j)}} = \delta_{ln}$ for $1 \leq n \leq 3$.
    By the trace theorem, there exist $\Phi_{j,l, \bn}, \Phi_{j,l, \bt} \in H^1(\Omega)$ satisfying
 \begin{alignat*}{2}
        \Phi_{j, l, \bn} |_{\Gamma_{k(j)}} &= \phi_{j, l} \cdot \bn_{k(j)}, \qquad & 
        \Phi_{j, l, \bt} |_{\Gamma_{k(j)}} &= \phi_{j, l} \cdot \bt_{k(j)}, \\
        \Phi_{j, l, \bn} |_{\Gamma \setminus \Gamma_{k(j)}} &= 0, \qquad & \Phi_{j, l, \bt}|_{\Gamma \setminus \Gamma_{k(j)}} &= 0, \\
        \| \Phi_{j, l, \bn} \|_{1, \Omega} &\leq C \|\phi_{j, l}\|_{1/2, \Gamma_{k(j)}}, \qquad & \| \Phi_{j, l, \bt} \|_{1, \Omega} &\leq C \|\phi_{j, l}\|_{1/2, \Gamma_{k(j)}},
 \end{alignat*}
    where $\bn_{k(j)}$ and $\bt_{k(j)}$ are the unit normal and tangent vectors on $\Gamma_{k(j)}$. Consequently, the symmetric tensor
 \begin{equation*}
        \Psi_{j, l} := \Phi_{j, l, \bn} \bn_{k(j)} \otimes \bn_{k(j)} + \Phi_{j, l, \bt} (\bt_{k(j)} \otimes \bn_{k(j)} + \bn_{k(j)} \otimes \bt_{k(j)}),  
 \end{equation*}
    satisfies $\Psi_{j, l} \in H^1(\Omega; \mathbb{S})$ with
 \begin{equation*}
        \int_{\Gamma_{N, j}} (\Psi_{j, l} \bn) \cdot r_n \ \ds = \delta_{ln}, \qquad 1 \leq n \leq 3, \quad \text{and} \quad  \| \Psi_{j, l} \|_{1, \Omega} \leq C \|\phi_{j, l}\|_{1/2, \Gamma_{k(j)}}.
 \end{equation*}
    Then $\Psi := \sum_{j=1}^{J_N} \sum_{l=1}^{3} \omega_{j, l} \Psi_{j, l}$ satisfies $\Psi \in H^1(\Omega; \mathbb{S}) \cap \hdivgamma{\Omega}{\mathbb{S}}$ and~\cref{eq:h1-rigid-interp}.
\end{proof}

 
\subsection{Inverting divergence}\label{sec:cont-inv-div}

The first step is to invert the divergence operator in the case where traction boundary conditions are imposed on the entirety of $\partial \Omega$:
 \begin{lemma}\label{lem:h1-inversion-div}
    For every $u \in \ltwomodrm(\Omega; \mathbb{V})$, there exists $\sigma \in H^1(\Omega; \mathbb{S}) \cap \hdivz{\Omega}{\mathbb{S}}$ satisfying
 \begin{equation}\label{eq:h1-inversion-div}
        \dive \sigma = u \quad \text{and} \quad \|\sigma\|_{1, \Omega} \leq C \|u\|_{0, \Omega}.
 \end{equation}
 \end{lemma}
\begin{proof}
    Let $u \in \ltwomodrm(\Omega; \mathbb{V})$. Let $R > 0$ be such that $\Omega \subset B_R$, and extend $u$ to be zero on $B_R \setminus \Omega$. By the Lax--Milgram lemma, there exists $\rho \in H_0^1(B_R; \mathbb{V})$ satisfying
 \begin{equation*}
        (\varepsilon(\rho), \varepsilon(\psi))_{B_R} = -(u, \psi)_{B_R} \qquad \psi \in H_0^1(B_R; \mathbb{V}).
 \end{equation*}
    Thanks to elliptic regularity \cite[Lemma 3.2]{Necas2011}, $\rho \in H^2(B_R; \mathbb{V})$ with $\|\rho\|_{2, B_R} \leq C \|u\|_{0, \Omega}$. 
    
    We now construct boundary corrections to the function $\varepsilon(\rho)$. Since $\varepsilon(\rho) \in H^1(\Omega_0; \mathbb{S})$ and $\dive \varepsilon(\rho) = u \in \ltwomodrm(\Omega_0; \mathbb{V})$, there exists $r_0 \in H^3(\Omega_0)$ satisfying
 \begin{equation*}
        (\airy r_0)\bn = \varepsilon(\rho) \bn \quad \text{on } \partial \Omega_0 \quad \text{and} \quad \|r_0\|_{3, \Omega_0} \leq C \| \varepsilon(\rho)\|_{1, \Omega_0} \leq C \|u\|_{0, \Omega_0}
 \end{equation*}
    by~\cref{lem:airy-correction-omega}. Let $\eta_0 \in C^{\infty}(\Omega_0)$ be any smooth function satisfying
 \begin{equation*}
        \eta_0 \equiv 1 \quad  \text{on } \partial \Omega_0 \quad \text{and} \quad \eta_0 \equiv 0 \quad \text{on } \Omega_i, \quad 1 \leq i \leq M.
 \end{equation*}
    Then the function $q_0 := \eta_0 r_0$ satisfies $q_0 \in H^3(\Omega)$ with
 \begin{equation*}
        (\airy q_0)\bn =  \varepsilon(\rho) \delta_{0j} \bn \quad \text{on } \partial \Omega_j, \quad 0 \leq j \leq M, \quad \text{and} \quad \|q_0\|_{3, \Omega_0} \leq C \|u\|_{0, \Omega}.
 \end{equation*}
    
     Moreover, for each $1 \leq i \leq M$, $\dive \varepsilon(\rho) \equiv 0$ on $\Omega_i$, and so the exactness of~\cref{eq:exact-sequence-continuous-disp} means that there exists $r_i \in H^2(\Omega_i)$ such that $\airy r_i = \varepsilon(\rho)|_{\Omega_i}$. Since $\varepsilon(\rho) \in H^1(\Omega_i; \mathbb{S})$, $r_i \in H^3(\Omega_i)$ with $\|\airy r_i\|_{1, \Omega_i} = \| \varepsilon(\rho) \|_{1, \Omega_i}$. Additionally, $r_i$ is unique up to affine functions, and thus we can choose $r_i$ such that $(r_i, \ell) = 0$ for all $\ell \in \mathcal{P}_1(\Omega)$ so that
 \begin{equation*}
        \|r_i\|_{3, \Omega_i} \leq C \|\airy r_i\|_{1, \Omega_i} = C \| \varepsilon(\rho) \|_{1, \Omega_i} \leq C \|u\|_{0,\Omega}.
 \end{equation*}
    Let $\tilde{r}_i$ denote the Stein extension~\cite{Stein1970} of $r_i$ to all of $\mathbb{R}^2$ which satisfies $\|\tilde{r}_i\|_{3, \mathbb{R}^2} \leq C \|r_i\|_{3, \Omega_i}$, and let $\eta_i \in C^{\infty}_c(\Omega_0)$ be any smooth function satisfying $\eta_i \equiv \delta_{ij}$ on $\Omega_j$, $1 \leq j \leq M$.
    Then the function $q_i := \eta_i \tilde{r}_i$ satisfies  $q_i \in H^3(\Omega)$,
 \begin{equation*}
        (\airy q_i)\bn = \varepsilon(\rho) \delta_{ij} \bn \quad \text{on } \partial \Omega_j, \quad 0 \leq j \leq M, \quad \text{and} \quad \|q_i\|_{3, \Omega} \leq C \|u\|_{0, \Omega}.
 \end{equation*}
    
    \noindent By construction, the tensor $\sigma := \varepsilon(\rho) - \airy \sum_{i=0}^{M} q_i$ then satisfies~\cref{eq:h1-inversion-div} and $\sigma \in \hdivz{\Omega}{\mathbb{S}}$, which completes the proof.    
\end{proof}

 \begin{remark}\label{rem:alt-proof-h1-inversion}
    The results in~\Cref{lem:h1-inversion-div} can be derived from the BGG machinery~\cite{Arnold2021, Cap2023}. In fact, the argument in~\cite{Arnold2021} implies that $\div: H^1(\Omega; \mathbb{S}) \cap \hdivz{\Omega}{\mathbb{S}} \to \ran(\div)$ is a bounded linear operator between Hilbert spaces and thus has bounded inverse. It is straightforward to see that $\ltwomodrm(\Omega; \mathbb{V})$ is the range of $\div$. Here we included a more explicit and constructive proof to show the role of the boundary conditions.
 \end{remark}

 
\subsection{Proof of~\cref{thm:h1-inversion-div-bc}}\label{sec:proof-of-h1-inversion-div-bc}

The case $|\Gamma_D| = |\Gamma|$ follows from~\cref{lem:h1-inversion-div}, so assume that $|\Gamma_D| < |\Gamma|$ and let $u \in L^2(\Omega; \mathbb{V})$ be given. For $1 \leq i \leq 3$, let $\Psi_i \in H^1(\Omega; \mathbb{S}) \cap \hdivgamma{\Omega}{\mathbb{S}}$ be given by~\cref{lem:h1-rigid-interp} with $\omega_{j,l} = \delta_{j0} \delta_{li}$  and define
$ v:= u - \sum_{i=1}^{3}  (u, r_i)_{\Omega} \dive \Psi_i.$
Then, there holds
 \begin{equation*}
    (v, r_i)_{\Omega} = (u, r_i)_{\Omega} - \sum_{j=1}^{3}  (u, r_j)_{\Omega} (\dive \Psi_j, r_i)_{\Omega} = (u, r_i)_{\Omega} - \sum_{j=1}^{3}  (u, r_j)_{\Omega} \delta_{ji} = 0,
 \end{equation*}
where we used the following identity for $\tau \in \hdiv{\Omega}{\mathbb{S}}$:
 \begin{equation*}
    \int_{\Omega} r \cdot \dive \tau \ \dx 
    = \int_{\Omega} \dive(r \cdot \tau) \ \dx = \int_{\Gamma} (r \cdot \tau) \cdot \bn \ \ds = \int_{\Gamma} (\tau \bn) \cdot r \ \ds \qquad \forall~r \in \rigidbodies. 
 \end{equation*}
Consequently, $v \in \ltwomodrm(\Omega; \mathbb{V})$, and thanks to~\cref{lem:h1-inversion-div}, there exists $\tau \in H^1(\Omega; \mathbb{S}) \cap \hdivz{\Omega}{\mathbb{S}}$ satisfying $\dive \tau = v$ and $\|\tau\|_{1, \Omega} \leq C \|v\|_{0, \Omega} \leq C \|u\|_{0, \Omega}$. The function $\sigma = \tau - \sum_{i=1}^{3}  (u, r_i)_{\Omega} \Psi_i$
then satisfies~\cref{eq:h1-inversion-div-bc}. 
\hfill \proofbox

 
\subsection{Inverting divergence with moments}

The last step combines the partial trace result~\cref{lem:h1-rigid-interp} with~\cref{thm:h1-inversion-div-bc} to invert the divergence operator with mixed boundary conditions and with specified moments on $\Gamma_N$.
 \begin{corollary}\label{cor:h1-inversion-div-bc-avg}
    For every $u \in L^2_{\Gamma}(\Omega; \mathbb{V})$ and $\omega \in \mathbb{R}^{J_N \times 3}$ satisfying~\cref{eq:omega-sum-conditions}, there exists $\sigma \in H^1(\Omega; \mathbb{S}) \cap \hdivgamma{\Omega}{\mathbb{S}}$ such that
 \begin{subequations}\label{eq:h1-inversion-div-bc-avg}
        \begin{alignat}{2}
            \dive \sigma &= u, \qquad & &\\
            \int_{\Gamma_{N, j}} (\sigma \bn) \cdot r_l &= \omega_{j, l}, \qquad & &1 \leq j \leq J_N, \ 1 \leq l \leq 3, \\
            \|\sigma\|_{1, \Omega} &\leq C \left( \|u\|_{0, \Omega} + |\omega| \right), \qquad & &
        \end{alignat}
 \end{subequations}
    where $C$ is independent of $u$ and $\omega$.
 \end{corollary}
\begin{proof}
    Let $u \in L^2_{\Gamma}(\Omega; \mathbb{V})$ and $\omega \in \mathbb{R}^{J_N \times 3}$ be given. By~\cref{thm:h1-inversion-div-bc}, there exists $\hat{\sigma} \in H^1(\Omega; \mathbb{S}) \cap \hdivgamma{\Omega}{\mathbb{S}}$ satisfying $\dive \hat{\sigma} = u$ and $\|\hat{\sigma}\|_{1, \Omega} \leq C \|u\|_{0, \Omega}$. Thanks to~\cref{lem:h1-rigid-interp}, there exists $\Psi \in H^1(\Omega; \mathbb{S}) \cap \hdivgamma{\Omega}{\mathbb{S}}$ such that
 \begin{align*}
        \int_{\Gamma_{N, j}} (\Psi \bn) \cdot r_l \ \ds &= \omega_{j, l} - \int_{\Gamma_{N, j}} (\hat{\sigma} \bn) \cdot r_l \ \ds, \qquad 1 \leq j \leq J_N, \\
        \|\Psi\|_{1, \Omega} &\leq C \left( |\omega| + \|\hat{\sigma}\|_{1, \Omega} \right). 
 \end{align*}
    Moreover, condition~\cref{eq:omega-sum-conditions} gives $(r_l, \dive \Psi)_{\Omega} = \sum_{j=1}^{J_N} \omega_{j, l} - (r_l, \dive \hat{\sigma})_{\Omega} = 0$.
    Consequently, there exists $\tau \in H^1(\Omega; \mathbb{S}) \cap \hdivz{\Omega}{\mathbb{S}}$ such that $\dive \tau = \dive \Psi$ and $\|\tau\|_{1, \Omega} \leq C \|\dive \Psi\|_{0, \Omega}$. The tensor $\sigma = \hat{\sigma} + \Psi - \tau$ then satisfies~\cref{eq:h1-inversion-div-bc-avg}.    
\end{proof}

 
\subsection{Proof of~\cref{lem:exact-sequence-continuous-disp}}\label{sec:proof-of-exact-sequence-continuous-disp}

The proof is similar to the discrete level, so we outline the proof. We first choose $\omega\in \mathbb{R}^{J_{N}\times 3}$ satisfying \cref{eq:proof:choosing-omegail} with $\tau \equiv 0$ and  $|\omega| \leq C |\kappa|$.        \Cref{cor:h1-inversion-div-bc-avg} shows that  there exists $\sigma \in H_{\Gamma}(\div, \Omega; \mathbb{S})$ such that $\dive \sigma \equiv 0$,
 \begin{equation*}
        \langle \sigma \bn, r_l \rangle_{\partial \Omega_m} = \sum_{ \substack{i=1 \\ \Gamma_{N, i} \subset \partial \Omega_m } }^{J_N} \omega_{i, l} = \kappa_{m, l} 
         \qquad \forall~m \in \mathfrak{I}^*, \ 1 \leq l \leq 3,
 \end{equation*}
    and $\|\sigma\|_{\dive, \Omega} \leq C |\omega| \leq C |\kappa|$. Then one can further decompose $\sigma$ into the image of $\airy$ and the harmonic part: $\sigma = \airy s + \phi$, where $\phi := \sigma - \airy s$ and
 \begin{equation*}
        s \in H^2_{\Gamma}(\Omega) : \qquad  (\airy s, \airy q) = (\sigma, \airy q) \qquad \forall~q \in H^2_{\Gamma}(\Omega).
 \end{equation*}
    The traces of $H^{2}_{\Gamma}(\Omega)$ ensure that $\langle (\airy s) \bn, r_l \rangle_{\partial \Omega_m} = 0$ for all $m \in  \mathfrak{I}^*$ and $1 \leq l \leq 3$, and so
    the harmonic part $\phi$ satisfies 
 \begin{subequations}\label{eq:harmonics}
        \begin{align}
            (\phi, \airy q) &= 0 \qquad & &\forall~q\in H^{2}_{\Gamma}(\Omega), \\
            (\dive \phi, u) &= 0 \qquad & &\forall~u\in L^{2}_{\Gamma}(\Omega; \mathbb{V}), \\
            \langle \phi \bn, r_l \rangle_{\partial \Omega_m} &= \kappa_{m, l} \qquad & &\forall~m \in \mathfrak{I}^*, \ 1 \leq l \leq 3,
        \end{align}
 \end{subequations}
    This shows that there exist functions $\phi\in H_{\Gamma}(\div, \Omega; \mathbb{S})$ satisfying~\cref{eq:harmonics}. Moreover, similar to the proof of~\cref{lem:projection-matching}, one can show that given $\kappa_{m, l} $, such $\phi$ is unique. This implies that the cohomological classes can be characterized by the value of $\langle \sigma \bn, r_l \rangle_{\partial \Omega_m}$, and therefore the dimension of the cohomology is $3|\mathfrak{I}^*|$.
\hfill\proofbox


\bibliographystyle{siamplain}
\bibliography{bib}
\end{document}